\documentclass[pdflatex,12pt,a4paper,twoside]{amsart}
\usepackage[a4paper,left=30mm,right=30mm,top=24mm,bottom=28mm]{geometry}
\numberwithin{equation}{section}
\usepackage{graphicx} 
\usepackage{amsmath, mathtools, amssymb, amsthm}
\usepackage{bbm}
\usepackage{xcolor}
\usepackage{subcaption}

\newtheorem{defi}{Definition}[section]
\newtheorem{thm}[defi]{Theorem}

\newtheorem{lemma}[defi]{Lemma}

\newtheorem{ass}[defi]{Assumption}
\newtheorem{cor}[defi]{Corollary}

\newtheorem{rem}[defi]{Remark}

\numberwithin{equation}{section}
\usepackage{enumitem}
\usepackage{crossreftools}

\renewcommand{\div}{\operatorname{div}}
\newcommand{\e}{\varepsilon}
\newcommand{\tx}{(t,x)}
\newcommand{\txy}{(t,x,y)}
\newcommand{\txsy}{(t,x,s,y)}
\newcommand{\txs}{(t,x,s)}
\newcommand{\txe}{\left(t,x,\tfrac{t}{\e},\tfrac{x}{\e}\right)}
\newcommand{\Oe}{{\Omega_\e}}
\newcommand{\Oes}{{\Omega_\e^\mathrm{s}}}
\newcommand{\Ge}{{\Gamma_\e}}
\newcommand{\Le}{\Lambda_\e}
\newcommand{\OeT}{{\Omega_\e^T}}
\newcommand{\OesT}{{\Oes^T}}
\newcommand{\Gtxs}{{\Gamma(t,x,s)}}

\newcommand{\psie}{\psi_\e}
\newcommand{\psiem}{\psi_\e^{-1}}

\newcommand{\Pe}{\Psi_\e}
\newcommand{\Je}{J_\e}
\newcommand{\Ae}{A_\e}
\newcommand{\Adj}{\operatorname{Adj}}
\newcommand{\Jem}{\Je^{-1}}
\newcommand{\Pem}{\Pe^{-1}}
\newcommand{\PemT}{\Pe^{-\top}}

\newcommand{\psin}{\psi_0}
\newcommand{\psinm}{\psi_0^{-1}}

\newcommand{\Pn}{\Psi_0}
\newcommand{\Pnm}{\Psi_0^{-1}}
\newcommand{\PnmT}{\Psi_0^{-\top}}
\newcommand{\Jn}{J_0}
\newcommand{\Jnm}{J_0^{-1}}
\newcommand{\An}{A_0}
\newcommand{\Anm}{A_0^{-1}}

\newcommand{\Aem}{A_\e^{-1}}

\newcommand{\PnT}{\Psi_0^{\top}}

\newcommand{\init}{\mathrm{in}}
\newcommand{\ue}{u_\e}
\newcommand{\ve}{v_\e}
\newcommand{\uein}{u_\e^\init}
\newcommand{\hue}{\hat{u}_\e}
\newcommand{\huein}{\hat{u}_\e^\init}
\newcommand{\fe}{f_\e}
\renewcommand{\ge}{g_\e}
\newcommand{\vGe}{v_\Ge}
\newcommand{\fn}{f_0}
\newcommand{\gn}{g_0}
\newcommand{\vGn}{v_{\Gamma}}
\newcommand{\hfe}{\hat{f}_\e}
\newcommand{\hge}{\hat{g}_\e}
\newcommand{\hvGe}{\hat{v}_\Ge}
\newcommand{\hfn}{\hat{f}_0}
\newcommand{\hgn}{\hat{g}_0}
\newcommand{\hvGn}{\hat{v}_{\Gamma}}
\newcommand{\vn}{v_0}
\newcommand{\wn}{w_0}
\newcommand{\un}{u_0}
\newcommand{\uo}{u_1}
\newcommand{\unin}{u_0^\init}
\newcommand{\hun}{\hat{u}_0}
\newcommand{\huo}{\hat{u}_1}
\newcommand{\hunin}{\hat{u}_0^\init}

\newcommand{\dd}{\, \mathrm{d}}
\newcommand{\dt}{\dd t}
\newcommand{\dx}{\dd x}
\newcommand{\ds}{\dd s}
\newcommand{\dy}{\dd y}
\newcommand{\xy}{(x,y)}

\newcommand{\dxt}{\dx \dt}
\newcommand{\dys}{\dy \ds}
\newcommand{\dyx}{\dy \dx}
\newcommand{\dsxt}{\ds \dx \dt}
\newcommand{\dysxt}{\dy \ds \dx \dt}

\newcommand{\Ee}{E_\e}

\newcommand{\N}{\mathbb{N}}
\newcommand{\Z}{\mathbb{Z}}
\newcommand{\R}{\mathbb{R}}

\newcommand{\Yp}{Y^*}
\newcommand{\Ypp}{{Y^*_\#}}

\newcommand{\Yptxs}{Y^*(t,x,s)}
\newcommand{\Ys}{{Y^\mathrm{s}}}
\newcommand{\Ysp}{{Y^\mathrm{s}_\#}}

\newcommand{\Ystxs}{Y^\mathrm{s}(t,x,s)}

\newcommand{\intT}{\int\limits_0^T}
\newcommand{\intS}{\int\limits_S}
\newcommand{\intO}{\int\limits_\Omega}
\newcommand{\intY}{\int\limits_Y}

\newcommand{\intYp}{\int\limits_{\Yp}}
\newcommand{\intG}{\int\limits_{\Gamma}}
\newcommand{\intSG}{\intS \intG}
\newcommand{\intYptxs}{\int\limits_{\Yptxs}}
\newcommand{\intGtxs}{\int\limits_{\Gtxs}}
\newcommand{\intOe}{\int\limits_\Oe}
\newcommand{\intTOe}{\intT \intOe}
\newcommand{\intTO}{\intT \intO}

\newcommand{\intOYp}{\intO \intYp}
\newcommand{\intTOS}{\intT \intO \intS}
\newcommand{\intTOSYp}{\intT \intO \intS \intYp}
\newcommand{\intTOSY}{\intT \intO \intS \intY}
\newcommand{\intTOSG}{\intT \intO \intS \intG}

\newcommand{\intSYp}{\intS \intYp}

\newcommand{\intTOSYptxs}{\intT \intO \intS \intYptxs}

\renewcommand{\hom}{\mathrm{hom}}
\newcommand{\Ke}{K_\e}
\newcommand{\one}[1]{{\bf{1}}_{#1}}
\newcommand{\1}{{\mathbbm{1}}}
\newcommand{\tss}[1]{\xrightarrow[]{\makebox[0.5cm]{$#1$}}	\hspace{-0.7cm}\xrightarrow[]{\makebox[0.6cm]{}}}
\newcommand{\tsw}[1]{{\xrightharpoonup[]{\makebox[0.5cm]{$#1$}}\hspace{-0.7cm}\xrightharpoonup[]{\makebox[0.6cm]{}}}}

\title[Effective advection induced by oscillating microstructure]{An effective advection induced by oscillating microstructure in a diffusion equation}
\author{David Wiedemann}
\address{Department of Mathematics\\
	Technical University of Dortmund, 44227 Dortmund, Germany.}
\email{david.wiedemann@math.tu-dortmund.de}

\subjclass[2020]{35B27, 80M40, 35R37}
\keywords{Homogenisation, oscillating microstructure, advection, diffusion equation, two-scale convergence}

\begin{document}
	\begin{abstract}
We consider the homogenisation of a diffusion equation in a porous medium.
The microstructure is time-dependent and oscillating on a small time scale.
This oscillation causes a novel advection in the homogenised equations. 
Allowing for a locally varying geometry, the oscillating microstructure demonstrates the ability to generate arbitrary and locally varying advection velocity fields.
	\end{abstract}
\maketitle
\section{Introduction}	
Transport of solutes in heterogeneous media is critical in various scientific domains like porous media flow, biological systems and chemical engineering. In materials with evolving microstructures, such as those undergoing pulsation or periodic deformation, the traditional diffusion equation may not fully capture the dynamics. Therefore, incorporating advective transport terms due to microstructural pulsation is crucial for accurately describing solute transport.

\subsection*{Goal of this work}
In this paper, we homogenise a diffusion equation within a porous medium featuring a time-pulsating pore structure. This dynamic microstructure induces a novel advection in the resulting homogenised equation, forming an advection--diffusion equation. The effective diffusivity and advection velocity are determined through the solutions of cell problems.
Specifically, we also examine scenarios where the pulsation of the microstructure varies across macroscopic space and time, rather than remaining constant. As a result, an effective advection velocity emerges, which varies both spatially and temporally. This presents an opportunity for inducing and actively controlling advective transport locally through modulation of the microstructure pulsation.

We consider a macroscopic domain $\Omega \subset \R^d$ for $d \geq 2$ on a macroscopic time interval $(0,T)$. The porous medium consists on two time-dependent phases $\Oe(t)$ and $\Oes(t)$ for $t \in (0,T)$. The domain $\Oe(t)$ with unit outer normal vector $\nu$ represents the in time and space oscillating pore space at time $t \in (0,T)$ where the diffusion happens and $\Oes(t)$ is the complementary solid space. The time pulsation interval shares the same scaling as the spatial periodicity cell, both being of the order of $\e$.
We represent their evolving interface as $\Ge(t)$ and denote the velocity of this boundary $\Ge(t)$ by $\vGe(t)$. We denote the remaining part of the boundary $\Oe(t)$, which is located at the boundary of $\Omega$ by $\Le(t)$. A precise formulation of the domain and the assumptions is given in Section~\ref{sec:MicroModel}.
For given initial values $\uein$ and source terms $\fe$ and $\ge$, we consider the homogenisation of
\begin{subequations}\label{eq:StrongForm}
	\begin{align}\label{eq:StrongForm:1}
		\partial_t u_\e - \div(D\nabla u_\e) &= \fe && \textrm{ in } \Oe(t) \, ,
		\\\label{eq:StrongForm:2}
		(- D \nabla u_\e - v_{\Ge} u ) \cdot \nu &= \e \ge && \textrm{ on }\Ge(t) \, ,
		\\\label{eq:StrongForm:3}
		(- D \nabla u_\e - v_{\Ge} u ) \cdot \nu &= 0 && \textrm{ on }\Le(t) \, ,
		\\\label{eq:StrongForm:4}
		\ue(0) &= \uein &&\textrm{ in } \Oe(0) \, ,
	\end{align}
\end{subequations}
where $\nu$ denotes the unit outer normal vector of $\Oe(t)$.
The term $-v_{\Ge} u \cdot \nu$ on the left-hand side of \eqref{eq:StrongForm:2} and \eqref{eq:StrongForm:3} accounts for the influence of the moving boundary and ensures mass conservation within $\Oe(t)$. In the case one models a solute diffusion, this term ensures that the solute at the interface moves along with the interface itself. The change in solute concentration by transport across the moving boundary $\Ge(t)$ is entirely determined by $\ge$. Specifically, if $\ge =0$, there is no exchange of solute at the interface, regardless of the movement of the domain.

\subsection*{Homogenisation result}
We show that the extension by zero of $\ue$ converges weakly to $u$, which is the unique solution of the following advection--diffusion equation:
\begin{subequations}\label{eq:Homogenised}
	\begin{align}
		\partial_t u -\div(D_\hom \nabla u - (W_\hom + V_\hom) u ) &= F + G && \textrm{in }(0,T) \times \Omega \, ,
		\\
		({D}_\hom \nabla u - ({W}_\hom + {V}_\hom)u)\cdot n & = 0 && \textrm{on }(0,T) \times \partial \Omega \, ,
		\\
		u(0) &= u^{\init} && \textrm{in } \Omega \, ,
	\end{align}
\end{subequations}
where $D_\hom$ denotes the effective diffusivity and $W_\hom$ and $V_\hom$ denote two advective transport velocities. Formulas for these coefficients are given in \eqref{eq:HomogenisedCoefficient} by means of the solutions of the cell problems \eqref{eq:CellProblem:I} and \eqref{eq:CellProblem:0}.
The bulk source terms $F$ and $G$ originate from the bulk source term $\fe$ and the source term $\e \ge$ at the interface, respectively.

\subsection*{Origin of the advection}
The homogenised equations \eqref{eq:Homogenised} contain the advection velocities $W_\hom$ and $V_\hom$. The term $W_\hom$ arises only for macroscopically modulated microstructures and vanishes if the porosity is macroscopically constant. It originates from a reformulation of the limit problem, which becomes useful for the separating the micro-and macroscopic variables in the general case.
A detailed analysis and physical explanation of this term is given in Section~\ref{subsec:SeparationTimeVariables}.
The more interesting advection is driven by $V_\hom$. It takes into account the transport arising from the pulsating microstructure. This velocity can be non zero even if $W_\hom$ vanishes. In the case of a strictly periodic microstructure it is a constant vector while a macroscopically modulated microstructure oscillation, leads to a time- and space dependent vector field. Thus local modulations of the microscopic oscillations can be employed for generating arbitrary advections.
A detailed analysis and discussion of te advective velocity $V_\hom$ including an example for which $V_\hom = 0$, while $W_\hom \neq 0$ is given in Section \ref{sec:Con:V} and illustrated in Figure \ref{fig:CellEvolution}.

\subsection*{Homogenisation approach}
For the homogenisation of \eqref{eq:StrongForm}, we use the two-scale transformation approach. We map the problem to a fixed periodically perforated domain $\e$ by means of a family of diffeomorphisms $\psie$. For the substitute problems, we pass to the homogenisation limit by means of two-scale convergence.
The resulting two-scale limit problem is defined on the macroscopic time--space domain $(0,T) \times \Omega$ times the cell time- and space- domain, i.e.~$(0,T) \times \Omega \times S \times \Yp$. In order to derive the homogenisation result for the actual two-scale limit domain, we transform the substitute pore domain $S \times \Yp$ by the two-scale limit diffeomorphisms $\psin(t,x,s, \cdot)$ to the actual reference pore domains $\bigcup\limits_{s \in S} \{s\} \times \Yptxs$. This approach is illustrated in Figure \ref{fig:TrafoMethodPulsating}. 

\begin{figure}[ht]
	\centering
	\includegraphics[width=1\linewidth]{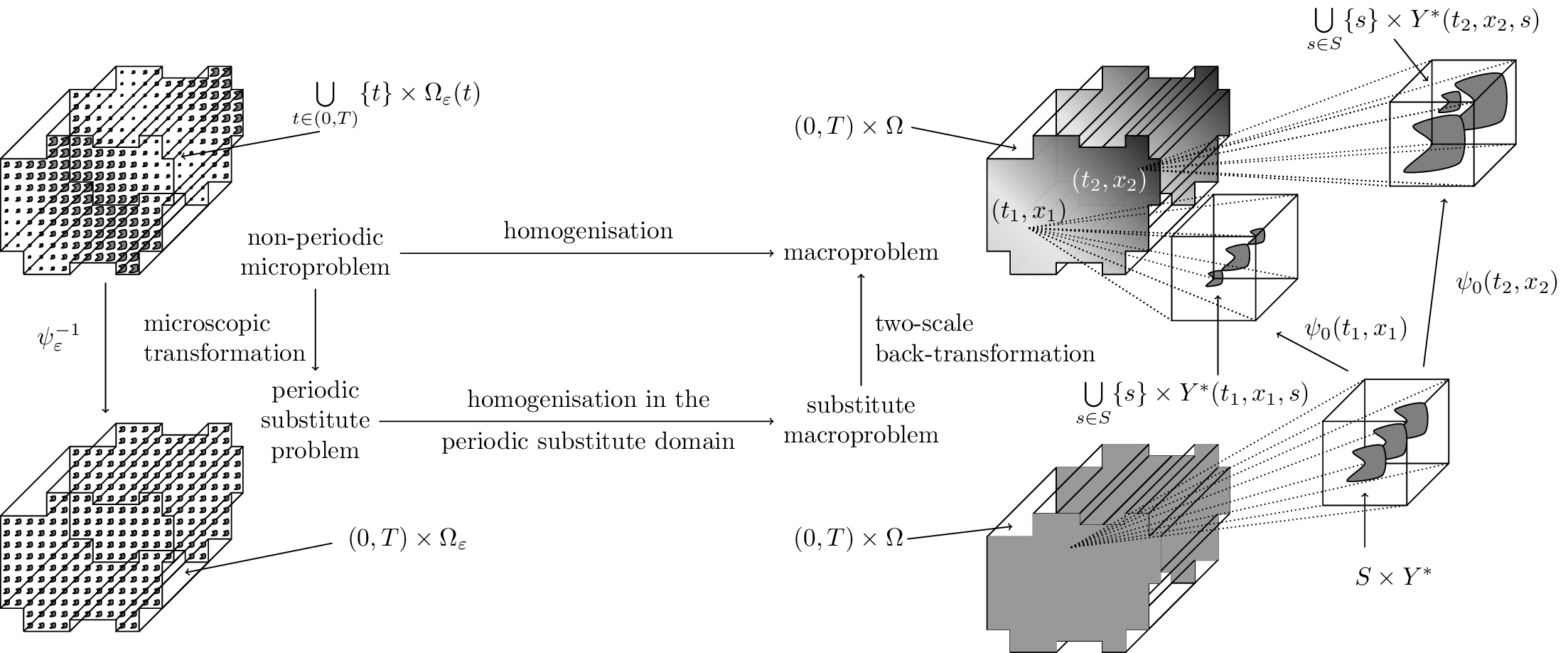}
	\caption{Two-scale transformation method for time-oscillating microstructure}
	\label{fig:TrafoMethodPulsating}
\end{figure}

In order to pass to the limit $\e \to 0$ in the substitute equations, we derive uniform a priori estimates for the solutions. Due to the time oscillations, one can not control uniformly the time derivative of $\hue$, which is the solution of the substitute equation. From a physical point of view one can also not expect that one can control it uniformly, since the oscillating porosity has to be compensated by some time oscillations of the concentration. Instead, we show a uniform estimate for the time derivative of the concentration $\hue$ times the local porosity.
We show a two-scale convergence compactness result fitting to those a priori estimates. Using this compactness result, we pass to the two-scale limit in the substitute equations.

In order to deduce the homogenised equations from the two-scale limit equation, we separate the micro- and macroscopic variables. For the separation of the spatial variables, we introduce effective coefficients and an advection term, which are determined by the solutions of cell problems. 
To separate the micro- and macroscopic time variable, we introduce an advective counter flow as described above.
For the transformation of the two-scale limit equations and the homogenised equations with its cell problems to the actual two-scale limit domain, we use the results and arguments of \cite{AA23,Diss} and introduce further two-scale transformation arguments.

\subsection*{Literature overview}
The homogenisation of linear and non-linear elliptic and parabolic problems in periodically perforated domains has been addressed in several works. See, for instance, \cite{CP79, Hor97} and the references therein. It enables not only the determination of effective coefficients but also the identification of new relevant terms in the effective equations e.g., the ``strange coming from nowhere'' \cite{CM82, CM97} or the Brinkman term \cite{All90}.	
Tools as the two-scale convergence \cite{All92, Ngu89, LNW02}, extension operators \cite{ACM+92} and the unfolding operator \cite{CDG08, CDG18} simplify the homogenisation for such kind of problems.

In many applications, chemical or biological reactions cause a slow change in the pore geometry. 
For given evolving microstructure, the homogenisation of (advection--) reaction--diffusion problems is considered in \cite{Pet07, Pet09, EM17, GNP21} the homogenisation of elasticity problems is considered in \cite{Ede19} and the homogenisation of Stokes flow is considered in \cite{JDE24, InstStokes, Diss}. There the problems are mapped to a fixed periodic microstructure in order to pass to the homogenisation limit for a substitute problem instead.
This approach is also used for the homogenisation of (advection--) reaction--diffusion problems in the case that the domain evolution is not given a priori but coupled with the unknowns \cite{GP23, NA23, GPPW24}.
This transformation approach is justified in \cite{AA23, Diss} by showing that the transformation and homogenisation commute in a meaningful way. Moreover, back-transformation arguments for the limit equations are derived that enable the derivation of transformation independent homogenisation results.

For a pulsating microstructure, only a few homogenisation results are available. In \cite{CP03}, the case of a rapidly pulsating microstructure is considered and extended to the case of random pulsation in \cite{CP06}. There the microscopic time interval length is of order $\e^2$ compared to the spatial microscopic scale of order $\e$. This scaling leads to an asymptotically large drift term such that the diffusion dynamics can be only observed in moving spatial coordinates and, thus, requires a macroscopically constant microstructure. There a non-mass-conserving boundary condition is considered leading also to a non-trivial drift term in the case where the homogenisation of our mass conserving model shows trivial advection. 

Parabolic problems with coefficients oscillating in both space and time were considered in \cite{BLP78}. The concept of two-scale convergence, featuring distinct scaling for oscillations in both time and space, is introduced in \cite{Hol97}. Within this framework, compactness results for two-scale convergence are established, facilitating the homogenisation of parabolic problems with oscillating coefficients in both temporal and spatial domains.

\subsection*{Structure of the article}
This article is organized as follows:
In Section~\ref{sec:MicroModel}, we formulate the $\e$-scaled parabolic problem and present the assumptions on the oscillating geometry. In Section~\ref{sec:Trafo}, we transform the problem into the substitute domain. In Section~\ref{sec:Existence}, we show the existence and uniqueness of a solution for this substitute problem and derive uniform a priori estimates. 
In Section~\ref{sec:Convergence}, we derive the convergence results by means of compactness arguments. In Section~\ref{sec:To-Scale-LimitIdent}, we pass to the limit in the parabolic equations by means of two-scale convergence leading to the two-scale limit equation. We separate the micro- and macroscopic variables leading to the homogenised equations in Section~\ref{sec:SeparationVariables}.
In Section~\ref{sec:Back-Trafo}, we transform the two-scale limit equations as well as the homogenised problems with its cell problems back to the actual two-scale domain. This leads to an transformation independent homogenisation result.
Finally, we investigate the new advection term, which is induces by the pulsating microstructure in Section~\ref{sec:Con:V}.

\subsection*{Notations}
Let $d \in \N$ and $U\subset \R^d$. For a function $u \colon U \to \R$ and a vector field $v \colon U\subset \R^{d}$, we use the following notation for its derivatives. For $x \in U$, we write $\nabla u(x) \in \R^d$ for the gradient of $u$ at $x\in U$, i.e.~$(\nabla u)_i(x) \coloneqq \partial_{x_i} u(x)$, and $\partial_x u(x) \coloneqq \nabla u^\top(x)\in \R^{1 \times d}$ for its transposed. We denote the Jacobian matrix of $v$ at $x \in U$ by $\partial_x v(x) \in \R^{d \times d}$ i.e.~$\partial_x v(x)_{ij} \coloneqq \partial_{x_j} v_i(x)$ and its transposed by $\nabla v(x) = \partial_x v^\top(x)$. 
Moreover, for $v \colon U\subset \R^{d}$, we define the divergence $\div v(x) = \sum_{i =1}^d\partial_{x_i} v_i(x)$.

We write $\1$ for the identity matrix and $\operatorname{Adj}(A)$ for the adjugate matrix of $A$, i.e.~$\operatorname{Adj}(A) A = \det(A) \1$.
With the above notation for derivatives, the Piola identity is written as $\div(\operatorname{Adj}(\partial_x v)) = 0$.

We use the index $\#$ to denote the periodicity of a function space, i.e.~for a domain $U \subset (0,1)^d$, $C_\#(U)$ denotes the subset of continuous functions on $\R^d$, which are $Y$-periodic, restricted to $U$. Similarly, we write $H^1_{\#}(U)$ to indicate the periodicity. 

We use $C>0$ as generic constant that can change during estimates but is independent of $\e$. 
For an measurable set $U$, we denote the indicator function by $\one{U}$, i.e.~$\one{U}(x) = 1$ if $x \in U$ and $\one{U}$, i.e.~$\one{U}(x) = 0$ if $x \not\in U$

\section{Microproblem}\label{sec:MicroModel}
In this section, we give a mathematical formulation of the $\e$-scaled problem and state the assumptions on the domain evolution and the data.

\subsection{Geometry}
The assumptions on the pulsating domain $\Oe(t)$ is given by means of a reference domain and a family of $\e$-scaled diffeomorphisms.
Let $\Omega \subset \R^d$ be an open domain representing the macroscopic domain and $[0,T]$ for $T \in \R$ the macroscopic time interval.
Let $(\e_n)_{n \in \N}$ be a positive sequence tending to zero as for instance $\e_n = n^{-1}$. In the following we write $\e = (\e)_{n \in \N}$.

\subsubsection*{Reference domain}
We assume that the macroscopic domain $\Omega \subset \R^d$ is open and bounded and consists on entire $\e$-scaled reference cells $Y = (0,1)^d$, i.e.~let $\Ke \coloneqq \{ k \in \Z^d \mid \e (k +Y) \subset \Omega \}$, we assume that $\Omega = \operatorname{int} \left(\bigcup\limits_{k \in \Ke} \e( k + \overline{Y}) \right)$.
Similarly, we assume that the time interval $(0,T)$ consists on entire $\e$-scaled reference time-intervals $S = (0,1)$.

We denote the open reference pore space in the periodicity cell by $\Yp \subset Y$ and its complementary solid part by $\Ys \coloneqq Y \setminus \Yp$. We denote the periodic extensions of $\Yp$ and $\Ysp$ by $\Ypp \coloneqq \operatorname{int} \left(\bigcup\limits_{k \in \Z^d} k + \overline{\Yp} \right)$ and $\Ysp \coloneqq \operatorname{int} \left(\bigcup\limits_{k \in \Z^d} k + \overline{\Ys} \right)$, respectively. We denote the interface of the pore and solid domain by $\Gamma\coloneqq \partial \Ypp \cap \partial \Ysp \cap \overline{Y}$.
We assume that $\Ypp$ is a Lipschitz domain.

The $\e$-scaled reference pore space $\Oe$, the $\e$-scaled reference solid space $\Oes$, their interface $\Ge$ and the reference outer boundary $\Le$ are given by
\begin{align*}
	\Oe \coloneqq \Omega \cap \e \Ypp \, , 
	\qquad 
	\Oes \coloneqq \Omega \cap \e \Ysp \, ,
	\qquad 
	\Ge \coloneqq \Omega \cap \partial \Oe \, ,
	\qquad 
	\Le \coloneqq \partial \Omega \cap \partial \Oe \, .
\end{align*}

\subsubsection*{Evolving microdomain}
In order to define the domains $\Oe(t)$, we use a family of mappings $\psie : [0,T] \times \overline{\Omega} \to \overline{\Omega}$. The assumptions on $\psie$ are given in Assumption~\ref{ass:psie}, below.
At time $t \in [0,T]$, we define the $\e$-scaled pore space $\Oe(t)$, the solid space $\Oes(t)$, their interface $\Ge(t)$ and the outer boundary $\Le(t)$ by
\begin{align*}
	\Oe(t) \coloneqq \psie(t, \Oe) \, , \quad 
	\Oes(t) \coloneqq \psie(t, \Oes) \, , \quad 
	\Ge(t) \coloneqq \psie(t, \Ge) \, , \quad 
	\Le(t) \coloneqq \psie(t, \Le) \, .
\end{align*}
We define the time-space sets by
\begin{align*}
	&\OeT \coloneqq \bigcup\limits_{t \in (0,T)} \{t\} \times \Oe(t) \, , \quad 
	&&\Omega_\e^{\mathrm{s}T} \coloneqq \bigcup\limits_{t \in (0,T)} \{t\} \times \Oes(t) \, , \\ 
	&\Gamma_\e^T \coloneqq \bigcup\limits_{t \in (0,T)} \{t\} \times \Ge(t) \, , \quad 
	&&\Le^T \coloneqq \bigcup\limits_{t \in (0,T)} \{t\} \times \Le(t) \, .
\end{align*}

\subsubsection*{Assumptions on the transformations}
We assume that the coordinate transformations $\psie$ fulfil the following regularity \ref{item:R1}--\ref{item:R2} and uniform bounds \ref{item:B1}--\ref{item:B3} as well as the exists some limit function \ref{item:L1}--\ref{item:L4} such that the asymptotic behaviour \ref{item:A1}--\ref{item:A5} is satisfied:

\begin{ass}[Assumptions on the transformations]\label{ass:psie}
	We assume that $\psie$ has the following regularity:
	\begin{enumerate}
		\item[{\crtcrossreflabel{(R1)}[item:R1]}] $\psie \in C^{1}([0,T];C^2(\overline{\Omega};\R^d))$\,,
		\item[{\crtcrossreflabel{(R2)}[item:R2]}] $\psie(t, \, \cdot \,)$ is a $C^2$-diffeomorphism from $\overline{\Omega}$ onto $\overline{\Omega}$ for every $t \in [0,T]$\,.
	\end{enumerate}
	We assume that $\psie$ satisfies the following uniform bounds:
	\begin{itemize}
		\item[{\crtcrossreflabel{(B1)}[item:B1]}] $\e^{j+l-1} \|\psie -x\|_{C^{j}([0,T]; C^l(\overline{\Omega}))} \leq C$ for $j \in \{0,1\}$ and $l \in \{0,1,2\}$\,,
		\item[{\crtcrossreflabel{(B2)}[item:B2]}] $\det(\partial_x\psie \tx ) \geq c_J$ for all $\tx \in [0,T] \times \Omega$ and some constant $c_J >0$\,,
		\item[{\crtcrossreflabel{(B3)}[item:B3]}] $\partial_x (\det(\partial_x\psie \tx )) \leq C$ for a.e.~$\tx \in [0,T] \times \Omega$ and some constant $c_J >0$\,. 
	\end{itemize}
	For the asymptotic behaviour of $\psie$, we assume that there exists a limit function $\psin$, which satisfies the following regularity 
	\begin{itemize}
		\item[{\crtcrossreflabel{(L1)}[item:L1]}] $\psin \in L^\infty((0,T) \times \Omega;C^1(S;C^2(\overline{Y}; \R^d)))$\,,
		\item[{\crtcrossreflabel{(L2)}[item:L2]}] $\psin(t,x,s,\cdot ) : \overline{Y} \to \overline{Y}$ is a for a.e.~$\tx \in (0,T) \times \Omega$ and every $s \in \overline{S}$ a $C^2$-diffeomorphism,
		\item[{\crtcrossreflabel{(L3)}[item:L3]}] the displacement mapping $y \mapsto \psin(t,x,y) -y$ can be extended $Y$-periodically i.e.~$(y \mapsto \psin(t,x,y) -y) \in L^\infty((0,T) \times \Omega;C^1(S;C^2_\#(\overline{Y}; \R^d)))$ \, ,
		\item[{\crtcrossreflabel{(L4)}[item:L4]}] the limit diffeomorphism can be extended $S$-periodically, i.e.~\newline
		$\psin \in L^\infty((0,T) \times \Omega;C^1_\#(S;C^2(\overline{Y}; \R^d)))$
	\end{itemize}
	and the following strong two-scale convergences hold
	\begin{itemize}
		\item[{\crtcrossreflabel{(A1)}[item:A1]}] $\e^{-1} (\psie\tx-x) \tss{2} \psin\txsy -y$\,,
		\item[{\crtcrossreflabel{(A2)}[item:A2]}] $\partial_x \psie \tss{2} \partial_y \psin$\,,
		\item[{\crtcrossreflabel{(A3)}[item:A3]}] $\partial_t \psie \tss{2} \partial_s \psin$\, ,
		\item[{\crtcrossreflabel{(A4)}[item:A4]}] $\e^{-1} (\psie(0,x)-x) \tss{2} \psin(0,x,0,y)-y$\,,
		\item[{\crtcrossreflabel{(A5)}[item:A5]}] $\partial_x \psie(0,x) \tss{2} \partial_y \psin(0,x,0,y)$\,.
	\end{itemize}
\end{ass}

The convergences in Assumption~\ref{ass:psie}\ref{item:A1}--\ref{item:A5} denote the strong two-scale convergence (see Definition \ref{def:two-scale}). Due to the uniform essential bounds, which are given by Assumption~\ref{ass:psie}\ref{item:B1}, the strong two-scale converges in Assumption~\ref{ass:psie}\ref{item:A1}--\ref{item:A5} hold also for arbitrary $p \in (1,\infty)$ instead of $2$.

The two-scale convergences of Assumption~\ref{ass:psie}\ref{item:A4}--\ref{item:A5} are formulated for a macroscopic domain $\Omega$ instead of $(0,T) \times \Omega$ and a reference cell $Y$ instead of $S \times Y$. They are used to identify for transforming and passing to the limit for the initial condition.

We use the following notation for the transformation quantities:
\begin{align}\label{eq:def:Pe}
	\Pe \coloneqq \partial_x \psie \, \qquad \Je\coloneq \det(\Pe) \, \qquad \Ae \coloneqq \operatorname{Adj}(\Pe) \, ,
	\\\label{eq:def:Pn}
	\Pn \coloneqq \partial_y \psin \, \qquad \Jn \coloneq \det(\Pn) \, \qquad \An \coloneqq \operatorname{Adj}(\Pn) \,.
\end{align}

We note that the above assumption ensure that $\Je \geq c_J$ and, thus, $\Pe$ is invertible and it holds $\Ae = \Je \Pem$. The uniform bound for $\Je$ from below can be transferred to $\Jn$ via the strong two-scale convergence of $\partial_x \psie$ and one gets $\Jn \geq c_J$ and, thus, also $\An = \Jn \Pnm$. 

For clarification, we note that the above uniform bounds for $\psie$ give
\begin{align*}
	\e^{-1} \| \psie -x \|_{L^\infty((0,T) \times \Omega)} + 
	\|\partial_x \psie \|_{L^\infty((0,T) \times \Omega)} + 
	\e \|\partial_x \partial_x \psie \|_{L^\infty((0,T) \times \Omega)} &\leq C \, ,
	\\
	\\|\partial_t \psie \|_{L^\infty((0,T) \times \Omega)} + \e
	\|\partial_x \partial_t\psie \|_{L^\infty((0,T) \times \Omega)} + 
	\e^2 \|\partial_x \partial_x \partial_t \psie \|_{L^\infty((0,T) \times \Omega)}& \leq C \,.
\end{align*}

\begin{rem}
	The Assumption~\ref{ass:psie} are used as follows: The regularity assumptions on $\psie$ allow us to transform the differential equation in the reference domain. 
	
	The uniform estimates on $\psie$ and its derivatives are crucial for the derivation of the uniform a priori estimates. Compared to the case of a slowly evolving domain, we can only assume the uniform bound for $\e^l \|\partial_t \psie\|_{C^{0}([0,T]; C^l(\overline{\Omega}))}$ instead of $\e^{l-1} \|\partial_t \psie\|_{C^{0}([0,T]; C^l(\overline{\Omega}))}$ for $l \in \{0,1,2\}$ in Assumption \ref{ass:psie}\ref{item:B1} due to the microscopic time oscillations.
	Moreover, Assumption \ref{ass:psie}\ref{item:B2} ensures that the Jacobian determinant $\Je$ is only varying macroscopically. Consequently, one has that $\Jn$ is constant with respect to $y$.
	
	The asymptotic behaviour of $\psie$ ensures that the coefficients in the transformed equations strongly two-scale converge and, hence, we can pass to the homogenisation limit. Moreover, it guarantees that the two-scale convergence of transformed and untransformed data and unknowns are equivalent (see \cite{AA23}) and thus, the homogenisation of the actual problem is equivalent to the homogenisation of the transformed problem. 
\end{rem}

\subsubsection*{Two-scale limit domain }
The two-scale limit domain of $\OeT$ should not be understood as domain in $(0,T) \times \Omega \times S \times Y$ but rather as family of domains $\Yptxs \subset Y$ with parameters $\txs \in (0,T) \times \Omega \times S$. In particular, for our homogenisation task it is not even necessary that $\Yptxs$ is defined for every $\tx \in (0,T) \times \Omega$. Indeed, it suffices that it is defined for a.e.~$\tx \in (0,T) \times \Omega$, where the null-set has to be chosen independent of the time $s \in S$. Nevertheless, at some points it simplifies the notation if one defines $\Yptxs$ for every $\tx \in (0,T) \times \Omega$ and defines the measurable set $\Omega_0^T$ by
\begin{align*}
	\Omega_0^T \coloneqq \bigcup\limits_{\txs \in (0,T) \times \Omega \times S} \{t\} \times \{x\}\times \{s\} \times \Yptxs\, .
\end{align*}
The set $\Omega_0^T$ and the domains $\Yptxs$ can be obtained by means of the two-scale convergence of the characteristic function of $\OeT$.
In this sense, the reference domain $(0,T) \times \Oe$ two-scale converges to $(0,1) \times \Yp$ for a.e.~$\tx \in (0,T) \times \Omega$.
The two-scale limit of the characteristic function $\one{(0,T) \times \Oe}$ is given by the function $\one{(0,T) \times \Omega \times S \times \Yp}$, which is an element in $L^p((0,T) \times \Omega \times S\times Y)$ and, thus, it does not define the domain uniquely. 
Indeed, for a.e.~$\txs \in (0,T) \times \Omega \times S$, $\one{\Yp}(t, x, s, \cdot)$ provides only the domain $\Yp \setminus N_1\txs \cup N_2\txs$ up to null sets $N_1\txs, N_2\txs \subset Y$.
The uniqueness can be obtained by requiring that for a.e.~$\txs \in \Omega$ the periodic extension of the domain is a Lipschitz domain. We
address the non-uniqueness of the two-scale limit representative of $\one{\OeT}$ in the same way. This provides the sets $\Yptxs$ for a.e.~$\tx \in (0,T) \times \Omega$. Lemma \ref{lem:TwoScaleEquiv} shows that
$\one{\OeT}\tx \tss{2,2} \one{(0,T) \times \Omega \times S \times \Yp}\txy$ if and only if $\one{(0,T) \times \Oe}(t,x,\psiem\tx) \tss{2,2} \one{(0,T) \times \Omega \times S \times \Yp}(t,x,s,\psinm\txsy)$. Thus, we can determine the two-scale limit for $\OeT$ and $\OesT$ by
\begin{align*} 
	\Yptxs &= \psin(t,x,s,\Yp) &&\textrm{ for every } s\in [0,1] \textrm{ and a.e.~}\tx \in (0,T) \times \Omega \, ,
	\\
	\Ystxs &= \psin(t,x,s,\Ys) &&\textrm{ for every } s\in [0,1] \textrm{ and a.e.~}\tx \in (0,T) \times \Omega\, .
\end{align*}
Their interface, is given by
\begin{align*}
	\Gtxs &= \psin(t,x,s,\Gamma) &&\textrm{ for every } s\in [0,T] \times[0,1] \textrm{ and a.e.~}\tx \in (0,T) \times \Omega \,
\end{align*}
and we define analogously to $\Omega_0^T$ the solid region $\Omega_0^{\mathrm{s}T}$ by
\begin{align*}
	\Omega_0^{\mathrm{s}T} \coloneqq \bigcup\limits_{\txs \in (0,T) \times \Omega \times S} \{t\} \times \{x\} \times\{s\} \times \Ystxs \,.
\end{align*}

For functions defined on $\OeT$ or $(0,T) \times \Oe$, we denote by $\widetilde{\cdot}$ the extension by $0$ to $(0,T) \times \Omega$. In the same way, for functions defined on $\Omega_0^T$ or $(0,T) \times \Omega \times S \times \Yp$, we denote by $\widetilde{\cdot}$ the extension by $0$ to $(0,T) \times \Omega \times S \times Y$. 

\subsection{Weak formulation of the $\e$-scaled problem}
The weak form of \eqref{eq:StrongForm} is given by:
Find $\ue \in L^2(\OeT)$ with $\nabla \ue \in L^2(\OeT)^d$ such that
\begin{align}\label{eq:WeakForm}
	\begin{aligned}
		-\int\limits_{\OeT} &\ue \partial_t \varphi+
		\int\limits_{\OeT} D \nabla \ue \cdot \nabla \varphi 
		+
		\int\limits_{(0,T)} \int\limits_{\Ge(t)} v_\Ge \cdot \nu u \varphi 
		\\
		&=
		\int\limits_{\OeT} \fe \varphi 
		-
		\int\limits_{(0,T)} \int\limits_{\Ge(t)} \e \ge \varphi + \int\limits_{\Oe(0)}\uein \varphi(0) \,
	\end{aligned}
\end{align}
for all $\varphi \in H^1(\OeT)$ with $\varphi(T, \cdot ) = 0$. 

\subsection{Assumptions on the data}
We make the following assumptions regarding the data:
\begin{ass}[Assumptions on the data]\label{ass:data}
	We assume that:
	\begin{itemize}
		\item the diffusivity coefficient is positive, i.e.~$D >0$,
		\item the velocity $\vGe$ of the boundary $\Ge(t)$ is given by 
		\begin{align}\label{eq:normalInterfaceVelocityEps}
			\vGe\tx = \partial_t \psie(t, \psiem\tx) \, .
		\end{align}
		
		\item $\fe$ is a bounded sequence in $L^2(\OeT)$, i.e.~
		\begin{align*}
			\|\fe\|_{L^2(\OeT)} \leq C \, ,
		\end{align*} 
		and there exists $\fn \in L^2(\Omega_0^T)$ such that
		\begin{align*}
			\fe \tsw{2} \fn \, .
		\end{align*}
		\item $(t \mapsto \|\ge(t)\|_{L^2(\Ge(t))}) \in L^2(0,T)$ with $\ge(t) \in L^2(\Ge(t))$ for a.e.~$t\in (0,T)$\,.
		\begin{align*}
			\e \intT \int\limits_{\Ge(t)} |\ge\tx|^2 \dd \sigma_x \, \dt \leq C
		\end{align*}
	and there exists $\gn\txs \in L^2(\Gtxs)$ with \newline $(\txs \mapsto \|g_0\txs\|_{L^2(\Gamma\txs)}) \in L^2((0,T) \times \Omega \times S)$ such that
		\begin{align*}
			\ge(t,\psie\tx) \tsw{2} \gn(t,x,s,\psin\txsy)
		\end{align*}
		in the sense of two-scale convergence on surfaces.
		\item 
		Let $\uein$ be a bounded sequence in $L^2(\Oe(0))$, i.e.~
		\begin{align*}
			\|\uein\|_{L^2(\Oe(0))} \leq C
		\end{align*}
		and there exists a function $\unin \in L^2(\Omega)$ such that
		\begin{align*}
			\uein(x) \tsw{2} \chi_{\Yp(0,x,0)}(y) \unin(x)
		\end{align*}
	\end{itemize}
\end{ass}

\begin{rem}
	The diffusivity constant $D$ can be generalized to a sequence of functions $D_\e \in L^\infty(\OeT)^{d \times d}$ for which a constant $\alpha >0$ exists such that
	$\xi^\top D\tx \xi \geq \alpha |\xi|^2$ for a.e.~$\tx \in \OeT$ and $D_0 \in L^\infty(\Omega_0^T)$ exists such that
	$\widetilde{D_\e}\tx \tss{p} \widetilde{D_0}\txsy$ for every $p \in (1,\infty)$, where $\widetilde{\cdot}$ denotes the extension by $0$ to $(0,T) \times \Omega$ and $(0,T) \times \Omega \times S \times Y$, respectively.
\end{rem}
\begin{rem}
	Note that the assumptions on the two-scale convergence of $\ge$ are given for the transformed quantity $\hge\tx = \ge(t,\psie\tx)$. Instead one can assume that $\ge$ can be extended to a function on $\Omega$ in $L^2(0,T;H^1(\Omega))$ and $\gn$ to a function in $L^2((0,T) \times \Omega \times S;H^1_\#(Y))$ such that
	\begin{align*}
		\ge \tsw{2} g_0 \,, \qquad \e \nabla_x \ge \tsw{2} \nabla_y g_0 \, .
	\end{align*} 
	Then, one has for the corresponding transformed function $\hge$ also the two-scale convergence
	\begin{align*}
		\hge \tsw{2} \hgn \,, \qquad \e \nabla_x \hge \tsw{2} \nabla_y \hgn \, ,
	\end{align*}
	which implies the two-scale convergence on surfaces.
\end{rem}

\section{Transformed microproblem}\label{sec:Trafo}
Using the mapping $\psie$, we can transform \eqref{eq:StrongForm} onto $[0,T] \times \Oe$ and the boundary conditions onto $[0,T] \times \Ge$ and $[0,T] \times \Le$, respectively.
We denote the transformed unknown and data by
\begin{align}\label{eq:TransfromationOfQuantities}
	\begin{aligned}
		&\hue\tx \coloneqq \ue(t,\psi\tx) && \textrm{ for } \tx \in (0,T) \times \Oe \, ,
		\\
		&\huein(x) \coloneqq \uein(\psi(0,x)) && \textrm{ for } \tx \in (0,T) \times \Oe \, ,
		\\
		&\hfe\tx \coloneqq \fe(t,\psi\tx) && \textrm{ for } \tx \in (0,T) \times \Oe \, ,
		\\
		&\hge\tx \coloneqq \ge(t,\psi\tx) && \textrm{ for } \tx \in (0,T) \times \Ge \, ,
		\\
		&\hvGe\tx \coloneqq \vGe(t,\psi\tx) && \textrm{ for } \tx \in (0,T) \times \partial \Oe \, .
	\end{aligned}
\end{align}
We denote the unit outer normal of $\Oe$ by $\hat{\nu}(x)$ and note the following identity 
\begin{align*}
	\| \PemT\tx \hat{\nu}\tx\|^{-1}\PemT \tx\hat{\nu}\tx = \nu(t,\psie\tx) \, ,
\end{align*}
which transforms the unit outer normal vector. The transformation of \eqref{eq:StrongForm} leads to the following transformed strong formulation:
\begin{subequations}\label{eq:StrongForm:Trafo}
	\begin{align}\label{eq:StrongForm:Trafo:1}
		\partial_t \hue - \partial_x \hue \Pem \partial_t \psie - \Jem\div(\Ae D \PemT \nabla \hue) &= \hfe && \textrm{ in }(0,T) \times \Oe \, ,
		\\\label{eq:StrongForm:Trafo:2}
		(-D\PemT \nabla \hue - \hue \hvGe) \cdot \|\PemT \hat{\nu}\|^{-1} \PemT \hat{\nu} &= \e \hge && \textrm{ on }(0,T) \times \Ge \, ,
		\\\label{eq:StrongForm:Trafo:3}
		(-D\PemT \nabla \hue - \hue \hvGe) \cdot \|\PemT \hat{\nu}\|^{-1} \PemT \hat{\nu} &= 0 && \textrm{ in }(0,T) \times \Le \, ,
		\\\label{eq:StrongForm:Trafo:4}
		\hue(t=0) &= \huein && \textrm{ in }\Oe \, ,
	\end{align}
\end{subequations}

One can derive the corresponding weak form by either transforming the integrals in the weak form \eqref{eq:WeakForm} or transforming directly the strong form \eqref{eq:StrongForm:Trafo}. 
For the latter approach, one has to multiply \eqref{eq:StrongForm:Trafo:1} by $\Je$.
Noting that the Jacobi formula together with the Piola identity gives $\partial_t \Je = \div(\Ae \partial_t \psie)$, we get with the product rule
\begin{align}\label{eq:BulkTermStrongFormTrafoModified}
	\partial_t (\Je \hue) - \div(\Ae \partial_t \psie \hue) - \div(\Ae D \PemT \nabla \hue) = \Je \hfe \,.
\end{align}
For the transformation of the boundary terms, we note that \eqref{eq:normalInterfaceVelocityEps} leads to 
\begin{align}\label{eq:normalInterfaceVelocityEps:Trafo}
	\hvGe\tx \cdot \PemT\tx \hat{\nu} = \partial_t \psie\tx \cdot \PemT\tx \hat{\nu} \, .
\end{align}
We multiply the boundary terms by $\Je \|\PemT \hat{\nu}\|$ and use the identity
\eqref{eq:normalInterfaceVelocityEps:Trafo}, which leads to
\begin{align}\label{eq:BoundaryConditionTrafoModified}
	\begin{aligned}
		(-\Ae D\PemT \nabla \hue - \hue \Ae \partial_t \psie) \cdot \hat{\nu} &= \e \Je \|\PemT \hat{\nu}\| \hge && \textrm{ on }(0,T) \times \Ge \,,
		\\
		(-\Ae D\PemT \nabla \hue - \hue\Ae \partial_t \psie) \cdot \hat{\nu} &=0 && \textrm{ on }(0,T) \times \Le \,.
	\end{aligned}
\end{align} 
We multiply \eqref{eq:BulkTermStrongFormTrafoModified} by a test function $\varphi$, integrate the divergence terms by parts and use the boundary conditions \eqref{eq:BoundaryConditionTrafoModified}, which leads to the following weak form of \eqref{eq:StrongForm:Trafo}:

\noindent Find $\hue \in L^2(0,T;H^1(\Oe))$ with $\partial_t \hue, \, \partial_t (\Je \hue) \in L^2(0,T;H^1(\Oe)')$ such that
\begin{subequations}\label{eq:WeakForm:Trafo}
	\begin{align}\label{eq:WeakForm:Trafo:1}
		\begin{aligned}
			\intT \langle \partial_t(\Je \hue)(t), &\varphi(t) \rangle_{H^1(\Oe)',H^1(\Oe)} \dt 
			+
			\int\limits_{(0,T) \times \Oe}
			\big(\Ae (\PemT \nabla \hue + \partial_t \psie) \hue \big) \cdot \nabla \varphi 
			\\
			&=
			\int\limits_{(0,T) \times \Oe} \Je \hfe \varphi
			-
			\int\limits_{(0,T) \times \Ge} \Je \|\PemT \hat{\nu}\| \e \hge \varphi \, ,
		\end{aligned}\\
		\begin{aligned}\label{eq:WeakForm:Trafo:2}
			\hue(0) &= \huein \hspace{8cm}
		\end{aligned}
	\end{align}
\end{subequations}
for all $\varphi \in L^2(0,T;H^1(\Oe))$.

\begin{lemma}\label{lem:EquiWeakForm}
	The weak forms \eqref{eq:WeakForm} are equivalent \eqref{eq:WeakForm:Trafo}, i.e.~$\ue$ is a solution of \eqref{eq:WeakForm:Trafo} if and only if $\hue$ is a solution of \eqref{eq:WeakForm:Trafo} and it holds $\hue\tx = \ue(t, \psie\tx)$ for a.e.~$\tx \in (0,T) \times \Oe$.
\end{lemma}
\begin{proof}
	Lemma \ref{lem:EquiWeakForm} follows directly from transforming the integrals \eqref{eq:WeakForm} in \eqref{eq:WeakForm:Trafo} and employing the identities \eqref{eq:TransfromationOfQuantities}.
\end{proof}

\section{Existence and a priori estimates}\label{sec:Existence}
In this section, we show that the weak form \eqref{eq:WeakForm:Trafo} has a unique solution. 
Moreover, we derive a priori estimates for $\ue$ as well as $\e$-independent a priori estimates for $(\Je\hue)$.
\begin{thm}\label{thm:ExistenceEps}
	For every $\e >0$, there exists a unique solution ${\hue \in L^2(0,T;H^1(\Oe))}$ with $\partial_t \hue, \, \partial_t (\Je \hue) \in L^2(0,T;H^1(\Oe)')$ of \eqref{eq:WeakForm:Trafo}.
	Moreover, there exists a constant $C>0$ such that
	\begin{align}\label{eq:est:Je-ue}
		\| \Je u_\e\|_{L^\infty(0,T;L^2(\Oe))} + \|\nabla (\Je u_\e)\|_{L^2((0,T) \times \Oe)} &\leq C \, ,
		\\\label{eq:est:dt:Je-ue}
		\|\partial_t(\Je u_\e)\|_{L^2(0,T;H^1(\Oe)')} &\leq C \,.
	\end{align}
\end{thm}

For the derivation of the uniform bounds in Theorem \ref{thm:ExistenceEps}, we note that the following uniform bounds for the coefficients in \eqref{eq:WeakForm:Trafo} hold.
\begin{lemma}\label{lem:unifromBoundsPe}
	Assume $\psie$ satisfies Assumption \ref{ass:psie}\ref{item:R1}--\ref{item:R2} and Assumption \ref{ass:psie}\ref{item:B1} for $j = 0$ and $l \in \{0,1\}$ and Assumption \ref{ass:psie}\ref{item:B2}. Let $\Pe$, $\Je$, $\Ae$ be given by \eqref{eq:def:Pe}.
	Then, there exist constants $\alpha, C>0$ such that
	\begin{align*}
		\| \Pe \|_{L^\infty((0,T)\times \Oe))}
		+
		\| \Je \|_{L^\infty((0,T)\times \Oe))}
		+
		\| \Ae \|_{L^\infty((0,T)\times \Oe)))}
		&\leq C,
		\\
		\| \Pem \|_{L^\infty((0,T)\times \Oe))} 
		+
		\| \Jem\|_{L^\infty((0,T)\times \Oe))}
		+
		\| \Aem \|_{L^\infty((0,T)\times \Oe)))}
		&\leq C \, ,
	\end{align*}
	\begin{align*}
		\|\xi^\top \Pem\tx D \PemT\tx \xi\| &\geq \alpha |\xi|^2 && \textrm{for all }\xi \in \R^d \, ,\\
		\|\xi^\top \Ae\tx D \PemT\tx \xi\| &\geq \alpha |\xi|^22 && \textrm{for all }\xi \in \R^d\, .
	\end{align*}
	If additionally Assumption \ref{ass:psie}\ref{item:B3} is satisfied, it holds
	\begin{align}
		\|\nabla_x \Je\|_{L^\infty((0,T)\times \Oe))} + \|\nabla_x \Jem\|_{L^\infty((0,T)\times \Oe))} \leq C \,.
	\end{align}
	
\end{lemma}
\begin{proof}
	The uniform estimate of $\Pe$ is given in Assumption \ref{ass:psie}\ref{item:B1} and the estimate for $\Jem$ by \ref{ass:psie}\ref{item:B1}. Since $\Je = \det(\Pe)$ and the entries of $\Ae = \Adj(\Pe)$ are polynomials in the entries of $\Pe$, we obtain also the uniform boundedness for $\Je$. 
	The entries of $\Pem$ are polynomials in the entries of $\Pe$ and $\Jem$ and, thus, it is also uniformly bounded. The entries of $\Aem = \Jem \Pe$ are polynomials in the entries of $\Pe$ and $\Jem$ and, thus, it is also uniformly bounded. 
	
	The uniform bound for $\nabla \Je$ is given by Assumption \ref{ass:psie}\ref{item:B3} and the estimate for $\nabla \Jem$ can be deduced with the uniform bound for $\Jem$ via the product rule, which gives $\nabla \Jem = -\Je^{-2} \nabla \Je$.
\end{proof}

For the transformed data, we obtain the following uniform bounds.
\begin{lemma}\label{lem:unifromBoundsData}
	Assume $\psie$ satisfy Assumption \ref{ass:psie}\ref{item:R1}--\ref{item:R2}, \ref{item:B1}--\ref{item:B2} and let $\fe$ and $\ge$satisfy Assumption~\ref{ass:data}.
	Then, for $\hfe$, $\hge$, $\huein$ given by \eqref{eq:TransfromationOfQuantities} there exist constant $C>0$ such that
	\begin{align*}
		\|\hfe \|_{L^2((0,T) \times \Oe)} +
		\e \|\hge \|_{L^2((0,T) \times \Ge)} +
		\|\huein\|_{L^2(\Oe)} \leq C \, .
	\end{align*}
	\begin{proof}
		Since $\Jem$ is uniformly essentially bounded, we get
		\begin{align*}
			\|\fe\|_{L^2((0,T) \times \Oe)}^2 &= \intTOe \hfe^2 = \int\limits_{\OeT} \Jem(t,\psie\tx) |f(t,x)|^2 \dxt
			\\
			&\leq C \int\limits_{\OeT} \Jem(t,\psie\tx) |f(t,x)|^2 \dxt = C \|\fe\|_{L^2(\OeT)}^2 \leq C\, .
		\end{align*}
		By a similar argument we obtain the bound for $\huein$ and with the uniform essential boundedness of $\Jem$
		and $\|\PemT \nu\|^{-1}$, we get also the uniform boundedness of $\huein$.
	\end{proof}
\end{lemma}
The following trace inequality for periodically perforated Lipschitz domains $\Oe$ is well known. It can be shown by a standard decomposition argument.
\begin{lemma}\label{lem:TraceInequality}
	Let $p \in [1,\infty)$. For every $\theta>0$ there exists a constant $C(\theta)>0$ independent of $\e$, such that for all $\ue \in W^{1,p}(\Oe)$
	\begin{align*}
		\e^{\frac{1}{p}} \|\ue\|_{L^p(\Ge)} \le C(\theta)\|\ue\|_{L^p(\Oe)} + \theta \e \|\nabla \ue\|_{L^p(\Oe)}.
	\end{align*}
\end{lemma}
\begin{proof}[Proof of Theorem \ref{thm:ExistenceEps}]
	\textbf{Existence:}
	The existence of a solution $\hue \in L^\infty(0,T;L^2(\Omega)) \cap L^2(0,T;H^1(\Oe))$ with ${\partial_t \hue , \partial_t (\Je \hue) \in L^2(0,T;H^1(\Oe)')}$ of \eqref{eq:WeakForm:Trafo} can be deduced by a standard Galerkin approach.
	
	\textbf{Uniform boundedness of $\hue$ and $\nabla \hue$:}
	Using the product rule, we obtain
	\begin{align*}
		\Ae D \PemT \nabla \hue 
		= \Ae D \PemT \nabla \Jem( \Je \hue) + \Pem D \PemT \nabla (\Je \hue).
	\end{align*}
	
	Thus, we can rewrite \eqref{eq:WeakForm:Trafo} as
	\begin{align}
		\begin{aligned}\label{eq:weakForm:Ju}
			\int\langle \partial_t (\Je u_\e )(t),& \varphi(t)\rangle_{H^1(\Oe)', H^1(\Oe)} \dt + (\Pem D \PemT \nabla (\Je \hue), \nabla \varphi)_{L^2((0,T) \times \Oe)} 
			\\
			&+ ((\Ae D \PemT \nabla \Jem + \Psi_\e^{-1}\partial_t\psi_\e) (\Je u_\e), \nabla \varphi)_{L^2((0,T) \times \Oe)} 
			\\
			=&(\Je\hfe, \phi)_{L^2((0,T) \times \Oe)}
			- (\Je \|\PemT \hat{\nu}\|\e \hge, \varphi)_{L^2((0,T) \times \Ge)} \, .
		\end{aligned}
	\end{align}
	Since $\Je \in L^\infty((0,T) \times \Oe)$, $\nabla \Je \in L^\infty((0,T) \times \Oe)^d$ and $\hue \in L^2(0,T;H^1(\Oe))$, we have $(\Je\hue) \in L^2(0,T;H^1(\Oe))$. Thus, we can test \eqref{eq:weakForm:Ju} with $\varphi = \chi_{(0,t)} \Je \hue$ for $t \in (0,T)$. After integrating the time derivative term by parts, we get
	\begin{align}
		\begin{aligned}\label{eq:TestedWithJu}
			\tfrac{1}{2}&\| (\Je\hue)(t)\|_{L^2(\Oe)}^2 + (\Pem D \PemT \nabla (\Je\hue), \nabla (\Je\hue))_{L^2((0,t) \times \Oe)}
			\\
			=& 
			\tfrac{1}{2}\| (\Je\huein)(0)\|_{L^2(\Oe)}^2
			-((\Ae D \PemT \nabla \Jem + \Pem \partial_t \psie )(\Je\hue), \nabla(\Je\hue) )_{L^2((0,t) \times \Oe)} 
			\\
			&+(\Je\hfe, \Je\hue )_{L^2((0,t) \times \Oe)}
			-
			(\Je \|\PemT \hat{\nu}\| \e \hge, \Je\hue)_{L^2((0,T) \times \Ge)} \, .
		\end{aligned}
	\end{align}
	We estimate the left-hand side of \eqref{eq:TestedWithJu} with the coercivity estimate for $\Pem D \PemT$ given by Lemma \ref{lem:unifromBoundsPe} from below. We estimate the right-hand side of \eqref{eq:TestedWithJu} from above using the uniform bounds of the coefficients and data, which are given in Lemma \ref{lem:unifromBoundsPe} and Lemma \ref{lem:unifromBoundsData}, respectively, and employing the H\"older inequality
	\begin{align*}
		\tfrac{1}{2}&\| (\Je\hue)(t)\|_{L^2(\Oe)}^2 + \alpha \|\nabla (\Je\hue)\|_{L^2((0,t) \times \Oe)} 
		\\
		\leq& C \|\huein\|_{L^2(\Oe)}^2 + C\|\Je\hue\|_{L^2((0,t) \times \Oe)}\|\nabla (\Je\hue)\|_{L^2((0,t) \times \Oe)}
		\\
		&+C \|\Je\hue\|_{L^2((0,t) \times \Oe)}
		+
		C \e^{1/2}\|\Je\hue\|_{L^2((0,t) \times \Ge)}\, .
	\end{align*}
	Using the uniform bound of $\huein$, the Young inequality and the $\e$-scaled trace inequality from Lemma \ref{lem:TraceInequality}, we obtain for every $\delta >0$ a constant $C_\delta$ such that
	\begin{align*}
		\tfrac{1}{2}\|J_\e(t) u_\e(t)& \|_{L^2(\Oe)} ^2
		+ \alpha \|\nabla(\Je u_\e)\|_{L^2((0,t) \times \Oe)}^2
		\\
		&\leq
		C + C_\delta \|\Je u_\e\|_{L^2((0,t) \times \Oe)}^2 + \delta \|\nabla (\Je u_\e)\|_{L^2((0,t) \times \Oe)}^2.
	\end{align*}
	After choosing $\delta$ small enough, we can compensate all the gradient terms on the right-hand side by means of the left-hand side. Then, the Lemma of Gronwall shows the uniform estimate \eqref{eq:est:Je-ue}.
	
	\textbf{Uniform boundedness of $\partial_t (\Je\hue)$:}
	Having the uniform estimate for $\Je \hue$ and $\nabla (\Je \hue)$, we obtain from Corollary \ref{cor:estimate:ue} uniform bounds for $\hue$ and $\nabla \hue$
	We rewrite \eqref{eq:WeakForm:Trafo}, use the H\"older inequality and inserting the uniform estimates for the coefficients as well as the uniform bounds for $\hue$ and $\nabla \hue$ and get
	\begin{align}
		\begin{aligned}
			| \langle \partial_t(&\Je \hue), \varphi \rangle_{L^2(0,T;H^1(\Oe)') , L^2(0,T;H^1(\Oe))} |
			\\
			\leq& |(\Ae D\PemT \nabla \hue, \varphi )_{L^2((0,T) \times \Oe)}| 
			+ |(\Ae\partial_t \psie \hue, \nabla \varphi )_{L^2((0,T) \times \Oe)}|
			\\
			& +	|(\Je\hfe, \nabla \varphi )_{L^2((0,T) \times \Oe)}|
			+|(\Je \|\PemT \hat{\nu}\| \e \hge, \varphi )_{L^2((0,T) \times \Ge)}|
			\\
			\leq& 
			C (\|\varphi\|_{L^2((0,T)\times \Oe)} + \|\nabla \varphi\|_{L^2((0,T)\times \Oe)})
			+ C(\e^{1/2}\|\varphi\|_{L^2((0,T)\times \Ge)}) 
			\\
			\leq& C \|\varphi\|_{L^2(0,T;H^1(\Oe))} \, ,
		\end{aligned}
	\end{align}
	which shows \eqref{eq:est:dt:Je-ue}.
\end{proof}
\begin{cor}\label{cor:estimate:ue}
	Let $\hue \in L^2(0,T;H^1(\Oe))$ be the solution of \eqref{eq:WeakForm:Trafo}.
	There exists a constant $C$ such that
	\begin{align}\label{eq:est:ue}
		\| u_\e \|_{L^2(0,T;H^1(\Oe))} \leq C.
	\end{align}
\end{cor}
\begin{proof}
	Employing the uniform essential boundedness of $J_\e^{-1}$ and \eqref{eq:est:Je-ue} as well as the uniform estimate on $\Je\hue$ given in \eqref{eq:est:Je-ue}, we obtain with the H\"older inequality
	\begin{align*}
		\begin{aligned}
			\| u_\e \|_{L^2((0,T)\times \Oe)} &= \|J_\e^{-1} \Je u_\e \|_{L^2((0,T)\times \Oe)} 
			\\
			&\leq \| J_\e^{-1} \|_{L^\infty((0,T)\times \Oe)} \| \Je u_\e \|_{L^2((0,T)\times \Oe)} \leq C \,.
		\end{aligned}
	\end{align*}
	Using also the uniform essential boundedness of $\nabla J_\e^{-1}$ and the uniform estimate on $\nabla ( \Je\hue)$ given in \eqref{eq:est:Je-ue}, we get
	\begin{align*}
		\begin{aligned}
			\| \nabla u_\e \|_{L^2((0,T)\times \Oe)} &= \| \nabla (J_\e^{-1} \Je u_\e) \|_{L^2((0,T)\times \Oe)} 
			\\&\leq \| \nabla J_\e^{-1} \Je u_\e)\|_{L^2((0,T)\times \Oe)} 
			+
			\| J_\e^{-1} \nabla (\Je u_\e))\|_{L^2((0,T)\times \Oe)} 
			\\
			&\leq C\| \Je u_\e\|_{L^2((0,T)\times \Oe)} 
			+
			C \| \nabla (\Je u_\e))\|_{L^2((0,T)\times \Oe)} \leq C \, .
		\end{aligned}
	\end{align*}
\end{proof}

\section{Convergence results}\label{sec:Convergence}
In order to pass to the homogenisation limit in \eqref{eq:WeakForm:Trafo}, we use the notion of two-scale convergence, which was introduced in \cite{All92, Ngu89} (see also \cite{LNW02}).

\begin{defi}[Two-scale convergence]
	\label{def:two-scale}
	Let $p\in [1,\infty)$ and $\tfrac{1}{p} + \tfrac{1}{q}=1$. A sequence $\ue \in L^p((0,T) \times \Omega)$ two-scale converges to $\un \in L^p((0,T) \times \Omega \times S \times Y))$ if
	\begin{align*}
		\lim\limits_{\e \to 0} \int\limits_{(0,T)} \intO \ue\tx \varphi \left(t, x, \frac{t}{\e}, \frac{x}{\e}\right) \dxt = \intTOSY \un\txsy \varphi\txsy \dysxt
	\end{align*}
	for all $\varphi \in L^{q}((0,T) \times \Omega;C_\#(S \times Y)))$.
	We write $\ue \tsw{p} \un$ or $\ue\tx \tsw{p} \un\txsy$ if we want to emphasize the dependence on the variables.
	For $p \in (1,\infty)$, we say that $\ue$ is strongly two-scale converging to $\un \in L^p((0,T) \times \Omega \times S \times Y)$ if and only if $\ue \tsw{p} \un$ and $\lim\limits_{\e \to 0} \|\ue\|_{ L^p((0,T) \times \Omega)} =
	\|\un\|_{L^p((0,T) \times \Omega \times S \times Y)}$. We write $\ue \tss{p} \un$, or $\ue\tx \tss{p} \un\txsy$ if we want to emphasize the dependence on the variables.
\end{defi}
The following two-scale convergence compactness results is well known.
\begin{lemma}
	Let $p \in (1,\infty)$ and $\ue$ a bounded sequence in $L^p((0,T) \times \Omega)$. Then, there exists a subsequence $\ue$ and $\un \in L^p((0,T) \times \Omega \times S \times Y)$ such that $\ue \tsw{p} \un$.
\end{lemma}

In order to derive a more subtle compactness result, we use the following extension operator $\Ee$.
\begin{lemma}\label{lem:ExtensionOperator}
	There exists a family of extension operators $\Ee : H^1(\Oe) \to H^1(\Omega)$ and a constant $C$ such that for all $\varphi \in H^1(\Oe)$
	\begin{align*}
		\|\Ee \varphi \|_{H^1(\Omega)} &\leq \|\varphi \|_{H^1(\Oe)} \, ,
		\\
		\|\Ee \varphi \|_{L^2(\Omega)} &\leq \|\varphi \|_{L^2(\Oe)} \, .
	\end{align*}
\end{lemma}

\begin{proof}
	Lemma \ref{lem:ExtensionOperator} was shown for spaces $H^1_{\textrm{loc}}$ and $L^2_{\textrm{loc}}$ in \cite{ACM+92} and extended for domains $\Omega$ that consists on entire $\e$-scaled cell to Lemma \ref{lem:ExtensionOperator} in \cite{Hoe16}.
\end{proof}

\begin{lemma}\label{lem:CompactnessStrongConvergence}
	Let $v_\e \in L^2(0,T;H^1(\Oe)) \cap H^1(0,T;H^1(\Oe)')$ be a bounded sequence, i.e.~there exists a constant $C>0$ such that
	\begin{align}
		\|v_\e \|_{L^2(0,T;H^1(\Oe))} + \|\partial_t v_\e \|_{L^2(0,T;H^1(\Oe)')} \leq C.
	\end{align}
	Then, there exists $v_0 \in L^2(0,T; H^1(\Omega)) \cap H^1(0,T;H^1(\Omega)')$, $v_1 \in L^2((0,T) \times \Omega \times S;H^1_\#(Y))$ and a subsequence $v_\e$ and such that
	\begin{align}
		\begin{aligned}\label{eq:CompactnessTheorem:v,grad}
			& \Ee v_\e \to v_0 && \textrm{ in } L^2((0,T) \times \Omega) \, ,
			\\
			&\nabla_x(\Ee v_\e)\tx \tsw{2} \nabla_x v_0\tx +\nabla_y v_1\txsy \, .
		\end{aligned}
	\end{align}
\end{lemma}
\begin{proof}
	Following \cite[Lemma 9]{GNK16}, we obtain for a.e.~$t \in (0,T)$ and $h\in (0,T)$ with $t+h \in (0,T)$,
	\begin{align*}
		\|\ve(t + h) -\ve\|_{L^2(\Oe)}^2 \leq \sqrt{h} \| \ve(t+h) -\ve(t) \|_{H^1(\Oe)} \|\partial_t \ve\|_{L^2(t,t+h, H^1(\Oe)')}
	\end{align*}
	After integration over $(0,T-h)$, we apply the H\"older inequality and get
	\begin{align}
		\begin{aligned}\label{eq:proof:ve+dtve<C:StrongCompact}
			\|\ve&(\cdot +h) -\ve(t) \|_{L^2(0,T-h;L^2(\Oe))}^2
			\\
			&\leq \sqrt{h} \int\limits_0^{T-h} \| \ve(t+h) -\ve(t) \|_{H^1(\Oe)} \|\partial_t \ve\|_{L^2(t,t+h;H^1(\Oe)')} \dt
			\\
			&\leq \sqrt{h} 2 \| \ve\|_{L^2(0,T;H^1(\Oe))} \Big( \int\limits_0^{T-h} \|\partial_t \ve\|_{L^2(t,t+h;H^1(\Oe)')}^2 \dt \Big)^{1/2}.
		\end{aligned}
	\end{align}
	By rearranging the order of integration in the last term, we can estimate it by
	\begin{align*}
		\int\limits_0^{T-h} \|\partial_t \ve\|_{L^2(t,t+h;H^1(\Oe)')}^2 \dt &= \int\limits_0^{T-h} \int\limits_t^{t+h} (\partial_t \ve(\tau))^2 \dd \tau \dt = \int\limits_0^{h} \int\limits_{\tau}^{T-h+\tau} (\partial_t \ve(t))^2 \dt \dd \tau 
		\\
		&\leq h\| \partial_t \ve\|_{L^2(0,T;H^1(\Oe)')}^2.
	\end{align*}
	Together with the boundedness of $\ve$ and $\partial_t \ve$, we can estimate \eqref{eq:proof:ve+dtve<C:StrongCompact} further and get
	\begin{align*}
		\|\ve(\cdot +h) -\ve \|_{L^2(0,T-h;L^2(\Oe))}^2 &\leq 2 h \| \ve\|_{L^2(H^1(\Oe))} \| \partial_t \ve\|_{L^2(0,T;H^1(\Oe)')} 
		\leq 2h C.
	\end{align*}
	Having this uniform convergence, Lemma \cite[Lemma 15]{NA23} provides the strong convergence of $\Ee \ve$ to some $v_0 \in L^2((0,T) \times \Omega)$.

	For some subsequence, the two-scale convergence of $\nabla (\Ee \ve)$ to $\nabla_x w_0 + \nabla_y w_1$ for some $w_0 \in L^2((0,T) \times S; H^1(\Omega))$, $w_1 \in L^2((0,T) \times \Omega \times S;H^1_\#(Y))$ can be shown as in \cite{All92} or \cite[Theorem 13]{LNW02}. Then, the two-scale limits $w_0$ and $v_0$ can be identified. Consequently $\wn$ is independent of $s \in S$ and $\vn \in L^2(0,T;H^1(\Omega))$
\end{proof}

\begin{lemma}\label{lem:CompactnessH1-Omega)}
	Let $v_\e \in L^2(0,T;H^1(\Oe))$ be a bounded sequence, i.e.~there exists a constant $C>0$ such that
	\begin{align}
		\|v_\e \|_{L^2(0,T;H^1(\Oe))} \leq C.
	\end{align}
	Then, there exists $v_0 \in L^2((0,T)\times S;H^1(\Omega))$, $v_1 \in L^2((0,T) \times \Omega \times S;H^1_\#(Y))$ and a subsequence $v_\e$ and such that
	\begin{align}
		\begin{aligned}\label{eq:CompactnessTheorem:v,grad:s-abh.}
			& \Ee v_\e \tx \tsw{2} v_0 \txs\, ,
			\\
			&\nabla_x(\Ee v_\e)\tx \tsw{2} \nabla_x v_0\txs +\nabla_y v_1\txsy \, .
		\end{aligned}
	\end{align}
\end{lemma}
\begin{proof}
	Lemma \ref{lem:CompactnessH1-Omega)} can be shown as in \cite{All92} or \cite[Theorem 13]{LNW02}.
\end{proof}

\begin{lemma}\label{lem:Conv:ue}
	Let $\Je$ by given by \eqref{eq:def:Pe} and $\hue$ as the solution of \eqref{eq:WeakForm:Trafo}. Then, there exists $u_0 \in L^2((0,T) \times (0,1) \times \Omega)$ with $\un \in L^2((0,T) \times (0,1) \times \Omega)^d$, $(\Jn \hun) \in L^2(0,T;H^1(\Omega))$ and $\hat{u}_1 \in L^2((0,T) \times \Omega \times S;H^1_\#(Y))$ such that for a subsequence
	\begin{subequations}
		\begin{align}
			&\Ee (\Je \hue) \tx \to (\Jn \hun) (t,x) \qquad \quad\textrm{ in } L^2((0,T) \times \Omega) \, ,
			\\
			&\widetilde{\nabla_x ( \Je \hue)} \tx = \one{\Yp}(y) \nabla_x (\Jn \hun) (t,x) + \one{\Yp}(y) \nabla_y (\Jn \hat{u}_1)(t,x,s,y) \, 
			\\
			&\widetilde{\ue}\tx \tss{2} \one{\Yp}(y) u_0(t,x,s) \, ,
			\\
			&\widetilde{\nabla_x \hue} \tx = \one{\Yp} \nabla_x \un(t,x,s) (t,x) + \one{\Yp} \nabla_y \hat{u}_1(t,x,s,y) \,
		\end{align}
	\end{subequations}
\end{lemma}
\begin{proof}
	We combine the a priori estimates for $\Je \hue$ given by \eqref{eq:est:Je-ue}--\eqref{eq:est:dt:Je-ue} with the compactness result Lemma~\ref{lem:CompactnessStrongConvergence} and
	the estimate on $\hue$ given by \eqref{eq:est:ue} with Lemma~\ref{lem:CompactnessH1-Omega)}. We obtain $v_0 \in L^2((0,T) \times \Omega)$ and $v_1 \in L^2((0,T) \times \Omega \times S;H^1_\#(Y)/\R))$ as well as $\hun \in L^2((0,T) \times S ;H^1(\Omega))$ and $\hat{u}_1 \in L^2((0,T) \times \Omega \times S;H^1_\#(Y))$ such that for a subsequence
	\begin{align*}
		&\Ee (\Je \hue) \tx \to v_0(t,x) \qquad \quad\textrm{ in } L^2((0,T) \times \Omega) \, ,
		\\
		&\nabla_x ( \Ee (\Je \hue)) \tx = \nabla_x v_0(t,x) + \nabla_y v_1(t,x,s,y) \, .
		\\
		&\Ee (\hue) \tx \tsw{2} \hun \txs \, ,
		\\
		&\nabla_x ( \Ee \hue) \tx = \nabla_x \hun (t,x,s) + \nabla_y \hat{u}_1(t,x,s,y) \, .
	\end{align*}
	To identify $\vn = \Jn \hun$ and $v_1 = \Jn \hat{u}_1$ in $\Yp$, we use the strong two-scale convergence of $\one{\Oe}$, which gives
	\begin{align*}
		\Je \widetilde{\hue} =
		\one{\Oe} \Ee (\Je \hue) \tss{2} \one{\Yp} \vn(t,x) \, ,
		\\
		\widetilde{\hue} =
		\one{\Oe} \Ee \hue \tss{2} \one{\Yp} \hat{u}_0(t,x,s) \,.
	\end{align*}
	With the strong two-scale convergence of $\Je \tss{2} \Jn$ and the uniqueness of the two-scale limit, we can identify $\vn = \Jn \hun$.
	
	Similarly, we get
	\begin{align*}
		\nabla_x \widetilde{\ue}\tx 
		&= \Je^{-1}\tx\widetilde{\nabla_x (\Je u_\e)}\tx + \nabla_x J_\e^{-1}\tx (\Je \widetilde{u_\e})\tx
		\\
		&\to \Jnm\txs \one{\Yp} \big( \nabla_x \vn\tx + \nabla_y v_1 \txsy \big)
		+
		\nabla_x \Jnm\txs \one{\Yp}\vn\tx
		\\
		&=
		\one{\Yp} \big(\nabla_x (\Jnm\txs \vn \tx) + \nabla_y (\Jnm v_1)\txsy \big)
	\end{align*}
	Consequently, $v_1 = \Jn \hat{u}_1$ in $\Yp$.
\end{proof}

\section{Identification of the limit problem}\label{sec:To-Scale-LimitIdent}
In this section, we pass to the limit $\e \to 0$ in \eqref{eq:WeakForm:Trafo}, which leads to the the following two-scale limit problem: 

Find $(\hun, \huo) \in L^2((0,T) \times S;H^1(\Omega)) \times L^2((0,T) \times \Omega \times S; H^1_\#(\Yp))$, with 
$(\Theta \hun) \in L^2(0,T;H^1(\Omega)) \cap H^1(0,T;H^1(\Omega)')$ such that
\begin{subequations}\label{eq:ts-limit-eq:transformed}
	\begin{align}
		\begin{aligned}\label{eq:ts-limit-eq:transformed:1}
			&\intT \langle \partial_t (\Theta \hun)(t), \varphi_0(t) \rangle_{H^1(\Omega)', H^1(\Omega)} \dt
			\\
			&+
			\intTOSYp \big( \An D \PnmT(\nabla_x \hun + \nabla_y \huo ) 
			+
			\An \partial_s \psin \hun\big)
			\cdot (\nabla_x \varphi_0 + \nabla_y \varphi_1) 
			\\
			=&
			\intTOSYp \Jn \hfn \varphi_0
			+
			\intTOSG \Jn \|\PnmT \hat{\nu}\| \hgn \, \varphi_0 \, ,
		\end{aligned}
		\\
		\begin{aligned}
			\label{eq:ts-limit-eq:transformed:2}
			(\Jn\hun)(0) &= \Jn(0,\cdot,0)\hunin \, \qquad \textrm{ in } L^2(\Omega) \,,\hspace{3.3cm}
		\end{aligned}
		\\
		\begin{aligned}\label{eq:ts-limit-eq:transformed:3}
			\Theta &\coloneqq |\Yp| \Jn \hspace{7cm}
		\end{aligned}
	\end{align}
\end{subequations}
for all $(\varphi_0, \varphi_1) \in L^2(0,T;H^1(\Omega)) \times L^2((0,T) \times \Omega \times S;H^1_\#(Y))$.

We note that $(\Theta \hun) \in L^2(0,T;H^1(\Omega))$ means in particular that $(\Theta \hun)$ is independent of $s \in S$, i.e.~it does not has oscillation with respect to the microscopic time $s \in S$. The microscopic time oscillations of $\Theta$ and $\hun$ cancel each other.

To pass to the limit $\e \to 0$ in \eqref{eq:WeakForm:Trafo}, we use the two-scale convergence of the transformed data and the transformation coefficients which is given in Lemma \ref{lem:Two-scale-conv:Data:Trafo} and Lemma~\ref{lem:Two-scale-conv:Coef}, respectively.

To pass to the limit $\e \to 0$ in \eqref{eq:WeakForm:Trafo}, we use the two-scale convergence of the transformed data and the transformation coefficients which is given in Lemma \ref{lem:Two-scale-conv:Data:Trafo} and Lemma~\ref{lem:Two-scale-conv:Coef}, respectively.
The two-scale convergence can be transferred from the untransformed sequence by the following equivalence.
\begin{lemma}\label{lem:TwoScaleEquiv}
	Let $p \in (1,\infty)$. Assume that $\psie$ and $\psin$ satisfy 
	\ref{ass:psie}\ref{item:R1}--\ref{item:R2}, \ref{item:B1}--\ref{item:B2}, \ref{item:L1}--\ref{item:L4}, \ref{item:A1}--\ref{item:A3}. Let $\ue \in L^p(\OeT)$, $\hue \in L^p((0,T) \times \Oe)$ with $\hue\tx = \ue(t,\psie\tx)$ and $\un \in L^p((0,T)\times \Omega \times S \times \Yp)$, $\hun \in L^p(\Omega_0^T)$ with $\hun\txsy = \un(t,x,s, \psin\txsy)$
	\begin{align*}
		\widetilde{\ue} &\tsw{p} \widetilde{\un}  && \textrm{ if and only if }&&\widetilde{\hue} \tsw{p} \widetilde{\hun} \, ,
		\\
		\widetilde{\ue} &\tss{p} \widetilde{\un}  && \textrm{ if and only if }&&\widetilde{\hue} \tss{p} \widetilde{\hun} \, . 
	\end{align*}
\end{lemma}
\begin{proof}
	The equivalence of the weak two-scale convergence is shown in \cite[Theorem 3.8]{AA23} and the equivalence of the strong two-scale convergence is shown in \cite[Theorem 3.14]{AA23}. Note that we apply the results of \cite{AA23} for the microscopic transformation $\tx \mapsto (t,\psie\tx)$ and the limit transformation $(\tx, (s,y)) \mapsto (\tx , (s, \psin\txsy))$.
\end{proof}
\begin{lemma}\label{lem:Two-scale-conv:Data:Trafo}
	Let $\fn$ and $\gn$ be given by Assumption \ref{ass:data}.
	Define 
	\begin{align*}
		&\hfn \tx \coloneqq \fn(t, \psin\tx) \textrm{ for a.e.~} \txsy \in (0,T) \times \Omega \times S \times \Yp \,,
		\\
		&\hgn \tx \coloneqq \gn(t, \psin\tx) \textrm{ for a.e.~} \txsy \in (0,T) \times \Omega \times S \times \Yp \,,
		\\
		&\hunin(x,y) \coloneqq \unin(\psin(0,x,0,y)) \textrm{ for a.e.~} \xy \in \Omega \times \Yp \,
	\end{align*}
	Then,
	\begin{align*}
		\widetilde{\hfe} \tsw{2} \widetilde{\hfn} \,, \qquad \qquad \hge \tsw{2} \hgn \, , \qquad \qquad \huein \tsw{2} \hunin
	\end{align*}
	where the two-scale convergence of $\hge$ holds in the sense of the two-scale convergence on surfaces (see for instance \cite{Neu96} and the two-scale convergence of $\huein$ holds for the macroscopic domain $\Omega$ instead of $(0,T) \times \Omega$ and the periodicity cell $Y$ instead of $S \times Y$.
\end{lemma}
\begin{proof}
The two-scale convergence of $\widetilde{\hfe}$ and $\huein$ is a direct consequence of Lemma \ref{lem:Two-scale-conv:Data:Trafo}. The two-scale convergence of $\ge$ is formulated in Assumption \ref{ass:data} in terms of the two-scale convergence of $\hge$.
\end{proof}

\begin{lemma}\label{lem:Two-scale-conv:Coef}
	Let $\psie$ and $\psin$ satisfy Assumption \ref{ass:psie} and let $\Pe$, $\Je$, $\Ae$ be given by \eqref{eq:def:Pe} and $\Pn$, $\Jn$, $\An$ be given by \eqref{eq:def:Pn}.
	Then, for every $p\in (1,\infty)$, it holds
	\begin{align*}
		&\Pe \tss{p} \Pn \, , \qquad
		\Je \tss{p} \Jn \, , \qquad
		\Ae \tss{p} \An \, , \qquad
		\Pem \tss{p} \Pnm \, ,
		\\
		&\Jem \tss{p} \Jnm \, , \qquad
		\Aem \tss{p} \Anm \, , \qquad
		\partial_t \psie \tss{p} \partial_s \psin
		\\
		&\Pe(0,x) \tss{p} \Pn(0,x,0,y) \, , \qquad
		\Je(0,x) \tss{p} \Jn(0,x,0,y) \, ,
	\end{align*}
	where the latter two two-scale convergences of $\Pe(0)$ and $\Je(0)$ are formulated for a macroscopic domain $\Omega$ instead of $(0,T) \times \Omega$ and a reference cell $Y$ instead of $S \times Y$.
\end{lemma}
\begin{proof}
	We note that the strong two-scale convergence of $\Pem$ and $\partial_t \psie$ is given by Assumption \ref{ass:psie}. To deduce the strong two-scale convergence of $\Je$ and $\Ae$, we rewrite them as polynomials with respect to the entries of $\Pe$. Following the proof of \cite[Lemma 3.3]{AA23}, we obtain the strong two-scale convergence of $\Jem$. Then, the strong two-scale convergence of $\Pem$, $\Aem$ follows from the fact that they are polynomials with respect to the entries of $\Pe$ and $\Jem$.
\end{proof}

\begin{thm}
	Let $\hue \in L^2(0,T; H^1(\Oe))$ be the solution of \eqref{eq:WeakForm:Trafo}. Then,
	\begin{subequations}\label{eq:ConvergenceUe}
		\begin{align}\label{eq:ConvergenceUe:1}
			&\Ee (\Je \hue) \tx \to (\Jn \hun) (t,x) \qquad \quad\textrm{ in } L^2((0,T) \times \Omega) \, ,
			\\
			&\widetilde{\nabla_x ( \Je \hue)} \tx = \one{\Yp}(y) \nabla_x (\Jn \hun) (t,x) + \one{\Yp}(y) \nabla_y (\Jn \hat{u}_1)(t,x,s,y) \, 
			\\
			&\widetilde{\ue}\tx \tss{2} \one{\Yp}(y) u_0(t,x,s) \, ,
			\\
			&\widetilde{\nabla_x \hue} \tx = \one{\Yp} \nabla_x \un(t,x,s) (t,x) + \one{\Yp} \nabla_y \hat{u}_1(t,x,s,y) \,
		\end{align}
	\end{subequations}	
	where $\hun \in L^2((0,T) \times S; H^1( \Omega))$ with $\nabla_s (\Jn \hun) = 0$ and $\huo \in L^2((0,T) \times \Omega \times S; H^1_\#(\Yp))$ with $(\Jn \huo) \in L^2((0,T) \times \Omega \times S; H^1_\#(\Yp))$ is given as the unique solution of \eqref{eq:ts-limit-eq:transformed}.
\end{thm}

\begin{proof}
	The convergences \eqref{eq:ConvergenceUe} are given by Lemma \ref{lem:Conv:ue}. It remains to identify $(\hun, \huo)$ as the solution of \eqref{eq:ts-limit-eq:transformed}.
	
	We test \eqref{eq:WeakForm:Trafo} with $\varphi\tx= \varphi_0\tx + \varphi_1\txe$ for $\varphi \in C^\infty([0,T];C^\infty(\overline{\Omega}))$ with $\varphi_0(T) = 0$ and $\varphi_1 \in D(0,T;C^\infty(\overline{\Omega}; C^\infty_\#(S \times Y)))$. We integrate the time derivative term by parts and insert the initial condition:
	\begin{align}
		\begin{aligned}\label{eq:WeakForm:Trafo-time-part-int}
			-\intTOe& \Je\tx \hue\tx \big(\partial_t\varphi_0\tx + \e \partial_t\varphi_1\txe + \partial_s\varphi_1\txe \big)\dxt
			\\
			&+
			\intOe\Je(0,x)\huein(,x) \varphi_0(0,x) 
			\dx
			\\
			&+
			\int\limits_{(0,T) \times \Oe} (\Ae D \PemT) \tx \nabla \hue\tx + \Ae\tx \partial_t \psie\tx \hue\tx 
			\\
			& \qquad \qquad \cdot \left(\nabla_x \varphi_0\tx + \e \nabla_x \varphi_1\txe +\nabla_y \varphi_1\txe \right) \dxt 
			\\
			=&
			\int\limits_{(0,T) \times \Oe} \Je\tx \hfe\tx \left(\varphi_0\tx +\e \varphi_1\txe \right)\dxt 
			\\
			&-
			\int\limits_{(0,T) \times \Ge} \Je\tx \|\PemT\tx \hat{\nu}\| \hge\tx \left(\varphi_0\tx +\e \varphi_1\txe \right) \dd \sigma_x \dt \, .
		\end{aligned}
	\end{align}
	Using the strong convergences of the coefficients and the data as well as the convergences of the unknowns given by \eqref{eq:ConvergenceUe}, we can pass to the limit in \eqref{eq:WeakForm:Trafo-time-part-int} and get
	\begin{align}
		\begin{aligned}\label{eq:proof:ident:ts-limit-transformed}
			&-\intTOSYp (\Jn \hun) \left(\partial_t \varphi_0\tx + \partial_s\varphi_1(t,x,s,y) \right)\dysxt
			\\
			&+
			\intOYp \Jn(0,x,0,y) \hunin(x) \varphi_0(0,x)+ \varphi_1(0,x,0,y) \dyx
			\\
			&+
			\intTOSYp \Big( (\An D \PnmT)\txsy(\nabla_x \hun \txs + \nabla_y \huo \txsy) 
			\\
			&\qquad \qquad +
			\An\txsy \partial_s \psin\txsy \Big)
			\cdot (\nabla_x \varphi_0\tx + \nabla_y \varphi_1\txsy) \dysxt
			\\
			=&\intTOSYp \Jn\txsy \hfn \txsy \varphi_0\tx \dysxt
			\\
			&+
			\intTOSG \Jn\| \PnmT \txsy \hat{\nu} \| \hgn \txsy \varphi_0\tx \dd \sigma_y \dsxt\,.
		\end{aligned}
	\end{align}
	Choosing $\varphi_1 = 0$ and $\varphi_0(0) = 0$ in \eqref{eq:proof:ident:ts-limit-transformed}, shows that $\partial_t (\Jn \hun) \in L^2(0,T; H^1(\Omega)')$. 
	We use again general test functions $\varphi_0$ and $\varphi_1$ as above and integrate the $t$-time derivatives in the first term of \eqref{eq:proof:ident:ts-limit-transformed} by parts. 
	For the $s$-derivative, we note that $(\Jn \hun)$ is independent of $s$. Moreover, $\varphi_1$ is $S$-periodic and, hence, this term is $0$. Substituting $\Theta \hun = |\Yp| \Jn = \intYp \Jn$ leads to \eqref{eq:ts-limit-eq:transformed}.
	
	By a density argument, \eqref{eq:ts-limit-eq:transformed} holds for all $\varphi_0 \in L^2(0,T;H^1(\Omega))$ and $\varphi_1 \in L^2((0,T) \times \Omega \times S;H^1(\Yp))$.
	
	Moreover, with Theorem \ref{thm:HomEqTrafoCoord} from below, we can identify $\hun$ with the unique solution of \eqref{eq:WeakForm:Homogenised:SubCoord}. Thus, the solution of \eqref{eq:ts-limit-eq:transformed} is unique and the convergence holds for the whole sequence. 
\end{proof}

\section{Separation of the micro- and macroscopic variables}\label{sec:SeparationVariables}
In this section we separate the micro- and macroscopic variables in \eqref{eq:ts-limit-eq:transformed}.
In a first step, we separate the spatial variables. This leads to \eqref{eq:PreHomogenised:SubCoord}, which still contains the microscopic time variable $s$. One can already see the advection term in \eqref{eq:PreHomogenised:SubCoord}, which arises from the transport which is induces by the pulsating microstructure.

In a second step, we separate also the micro- and macroscopic time variables in \eqref{eq:PreHomogenised:SubCoord}. However, the diffusive flux is proportional to $\nabla \hun$, which depends on the microscopic time variable. To circumvent this issue, we rewrite the equation in terms of $(\Theta \hun)$, where $\Theta$ is the local porosity. Since $(\Theta \hun)$ is independent of the microscopic time variable, this allows us to separate the micro- and macroscopic time variables. In order to rewrite the diffusive flux in terms of $\nabla (\Theta\hun)$, we introduce an second artificial transport term. The result is \eqref{eq:Homogenised:SubCoord}.

\subsection{Separation of the spatial variables}
The separation of the spatial macro- and microscopic variables leads to the following system:
\begin{subequations}\label{eq:PreHomogenised:SubCoord}
	\begin{align}
		\partial_t (\Theta \hun) - \div( \hat{D}^* \nabla \hun - \hat{V}^* \hun) &= \hat{F}^* + \hat{G}^*&& \textrm{in } (0,T) \times \Omega \times S \, ,
		\\
		(\hat{D}^* \nabla \hun - \hat{V}^* \hun) \cdot n &= 0 && \textrm{in } (0,T) \times \partial \Omega \times S \, ,
		\\
		\partial_s (\Theta \hun) &= 0 &&\textrm{in } (0,T) \times \Omega \times S \, ,
		\\
		y &\mapsto \hun\txsy &&Y\textrm{-periodic} \, ,
		\\
		\hun(0) &= \hun^{\init} &&\textrm{in } \Omega \,.
	\end{align}
\end{subequations}
where the porosity $\Theta$, the effective diffusion coefficient $\hat{D}^*$, the new advection velocity $\hat{V}^*$ and the bulk source terms $\hat{F}$ and $\hat{G}$ are given by
\begin{subequations}\label{eq:PreHomogenisedCoeffcients:SubCoord}
	\begin{align}
		&\Theta\txs \coloneqq \int\limits_{\Yp} \Jn\txs \dy = |\Yp| \Jn\txs \,,
		\\
		&\hat{D}_{ij}^*\txs \coloneqq \int\limits_{\Yp} (\An D \PnmT)\txsy (e_j + \nabla_y \hat{\zeta}_j\txsy) \cdot e_i \dy \,,
		\\\label{eq:PreHomogenisedCoeffcients:SubCoord:V}
		&\hat{V}^*\txs \coloneqq -\intYp (\Jn D \PnmT)\txsy \nabla_y \hat{\zeta}_0\txsy \dy \,,
		\\
		&\hat{F}\tx \coloneqq \intSYp \Jn\txsy \hfn\txsy \dys \,,
		\\
		&\hat{G}\tx \coloneqq \intSG \Jn\txsy \| \PnmT\txsy \hat{\nu} \txsy \|\hgn \txsy \dd \sigma_y \ds \,.
	\end{align}
\end{subequations}
The effective diffusion coefficient $\hat{D}^*$ and the advection velocity $\hat{V}^*$ depend on the solutions $\hat{\zeta}_i$ for $i \in \{1,\dots, n\}$ and $\hat{\zeta}_0$ of the following cell problems:
\begin{align}\label{eq:CellProblem:I:SubCoord}
	\begin{aligned}
		-\div( \An D \PnmT(e_j + \nabla_y \hat{\zeta}_j)) &= 0 &&\textrm{in } \Yp\,,
		\\
		\An D \PnmT(e_j + \nabla_y \hat{\zeta}_j) \cdot \hat{\nu} &= 0 &&\textrm{on } \Gamma \,,
		\\
		y &\mapsto \hat{\zeta}_j &&Y-\textrm{periodic}
	\end{aligned}
\end{align}
and 
\begin{align}\label{eq:CellProblem:0:SubCoord}
	\begin{aligned}
		-\div( \An (D \PnmT \nabla_y \hat{\zeta}_0 + \partial_s \psin)) &= 0 &&\textrm{in } \Yp \,,
		\\
		\An (D \PnmT \nabla_y \hat{\zeta}_0 + \partial_s \psin) \cdot \hat{\nu} &= 0 &&\textrm{on } \Gamma \,,
		\\
		y &\mapsto \hat{\zeta}_0 &&Y-\textrm{periodic}.
	\end{aligned}
\end{align}

The weak form of \eqref{eq:PreHomogenised:SubCoord} is given by: 
Find $\hun \in L^2((0,T) \times S; H^1(\Omega))$ with $\partial_t (\Theta \hun) \in L^2(0,T;H^1(\Omega)')$ such that
\begin{subequations}
	\begin{align}
		\label{eq:WeakForm:PreHomogenised:SubCoord}
		\begin{aligned}
			&\intTO \partial_t (\Theta \hun)\tx \varphi\tx \dxt 
			\\
			&+ \intTOS \hat{D}^*\txs \nabla \hun\txs - \hat{V}^*\txs \hun \txs \ds \cdot \nabla_x \varphi\tx \dxt 
			\\
			&=
			\intTO (\hat{F}\tx + \hat{G} \tx ) \,\varphi\tx \dxt \,,
		\end{aligned}
		\\
		\begin{aligned}
			&\intS \Theta \txs \hun \txs \partial_s \varphi_1(s) \ds= 0 \, , \hspace{5cm}
		\end{aligned}
		\\
		\begin{aligned}
			&(\Theta \hun)(0) = \Theta(0) \hunin \, \hspace{7cm}
		\end{aligned}
	\end{align}
\end{subequations}
for all $\varphi \in L^2(0,T;H^1(\Omega)$ and $\varphi_1 \in D((0,1))$.

The weak forms for the cell problems \eqref{eq:CellProblem:I:SubCoord} and \eqref{eq:CellProblem:0:SubCoord} are given by \eqref{eq:WeakForm:CellProblem:I:SubCoord} and \eqref{eq:WeakForm:CellProblem:0:SubCoord}, respectively:

For $j \in \{1, \dots, d\}$, find $\hat{\zeta}_j \in L^\infty((0,T) \times \Omega \times (0,1);H^1_\#(\Yp)/\R)$ such that
\begin{align}\label{eq:WeakForm:CellProblem:I:SubCoord}
	\intYp (\An D \PnmT)\txsy(e_j + \nabla_y \hat{\zeta}_j\txsy) \cdot \varphi(y) \dy &= 0 
\end{align}
for all $\varphi \in H^1_\#(\Yp))$ and a.e.~$\txs \in (0,T) \times \Omega \times S$.

Find $\hat{\zeta}_0 \in L^\infty((0,T) \times \Omega \times (0,1);H^1_\#(\Yp)/\R)$ such that
\begin{align}\label{eq:WeakForm:CellProblem:0:SubCoord}
	\intYp (\An\txsy(\partial_s \psin\txsy + D \PnmT \nabla_y \hat{\zeta}_0\txsy)) \cdot \varphi(y) \dy &= 0 
\end{align}
for all $\varphi \in H^1_\#(\Yp))$ and a.e.~$\txs \in (0,T) \times \Omega \times S$.

\subsection{Separation of the time variables}\label{subsec:SeparationTimeVariables}

In this section, we separate the micro- and macroscopic time variables by introducing the counter advection $W_\hom$.

To understand the origin and the meaning of the two advective transport velocities, we note that the two-scale limit $\un$ of $\ue$ can contain oscillations on the reference time oscillation interval $S=(0,1)$. Thus, one can not directly formulate the resulting advection--diffusion equation in terms of only $\un$ solely on the macroscopic time--space domain $(0,T) \times \Omega$ but has to use the macroscopic space, macro- and microscopic time domain $(0,T) \times \Omega \times S$ instead. This equation is given by \eqref{eq:PreHomogenised}. Compared to the two-scale limit, the weak limit $u = \intS\intYptxs \un \dy = \intS |\Yptxs| \un\txs \dy $ of $\ue$ averages over the time interval $S$ and compensates the oscillation of $\un$ by the oscillations of the pore volume $|\Yptxs|$.
Both quantities $\un$ and $u$ have physical relevant meanings. For a solute diffusion process $\un$ represents the concentration of the solute, while $u$ represents the effective local mass of a concentration, since it averages over the porosity. The homogenised equation naturally depends on both quantities, since the time-derivative is formulated in terms of $u$, while the diffusive flux is proportional to the concentration $\un$ with a Diffusion coefficient depending on the microstructure $\Yptxs$.
For a fixed periodic microstructure with reference pore space $\Yp$ and porosity $\Theta = |\Yp|$, this issue can be easily tackled by rewriting the effective diffusive flux by~$-D^* \nabla \un = -(\Theta^{-1} D^*) \nabla u$ for $u = \Theta \un$. In the case of macroscopically varying porosity this reformulation causes an artificial advective counterflow term which proportional to the gradient of the inverse of the porosity, i.e.~
\begin{align}\label{eq:ExplanationAdvectiveCounterflow}
	-D^* \nabla \un = -(\Theta^{-1} D^*) \nabla u - D^* \nabla (\Theta^{-1}) u \eqqcolon -D_\hom \nabla u +  W_\hom u \,.
\end{align}
In the case of an slowly evolving and not pulsating microstructure, this reformulation becomes not necessary in order to formulate the homogenised equations. However, in the case of a pulsating microstructure where also the porosity $\Theta\txs = |\Yptxs|$ is oscillating with respect to the times $s \in S$, only $u$ and not $\un$ are independent of the microscopic time interval. Thus, the reformulation \eqref{eq:ExplanationAdvectiveCounterflow} and the introduction of the artificial advective counter flow becomes highly useful.

In case the geometry of the microstructure is oscillating with respect to time while the porosity is only changing with respect to the macroscopic time but not oscillating with respect to the microscopic time, this introduction of the counter flow is not necessary and one can derive also the homogenised system \eqref{eq:PreHomogenised:NoPulsTheta}, instead.

We define
\begin{align*}
	u \coloneqq \Theta \hun \, , \qquad u^\init \coloneqq \Theta(0, \cdot, 0) \hunin \, .
\end{align*} 
Having the convergence $E_\e (\Je \hue) \to \Jn \hun$ given by \eqref{eq:ConvergenceUe:1}, we obtain in particular the strong two-scale convergence 
$E_\e (\Je \hue) \tss{2} \Jn \hun$. Together with the strong two--scale convergence of $\one{\Oe}$, we obtain $\widetilde{\Je \hue} = \one{\Oe}E_\e (\Je \hue) \tss{2} \one{\Yp} \Jn \hun$. Taking the $S \times Y$-integral over the two-scale limit, the two-scale convergence implies the weak convergence
\begin{align*}
	\widetilde{\Je \hue} \rightharpoonup \intS \intY \one{\Yp} \Jn \hun = \intY \one{\Yp} \Jn \hun= |\Yp|\Jn \hun = \Theta u_0 = u \,.
\end{align*}

We note that $u$ is independent of $s\in (0,T)$. 
The product rule gives
\begin{align*}
	\hat{D}^* \Theta^{-1} \nabla_x u + \hat{D}^* \nabla_x(\Theta^{-1})u 
	=\hat{D}^* \nabla_x (\Theta^{-1} u) =
	\hat{D}^* \nabla_x \hun 
	\,.
\end{align*}
Inserting this in \eqref{eq:PreHomogenised:SubCoord}, gives
\begin{subequations}\label{eq:Homogenised:SubCoord}
	\begin{align}
		\partial_t u -\div(\hat{D}_\hom \nabla u - (\hat{W}_\hom + \hat{V}_\hom) u ) &= \hat{F} + \hat{G} && \textrm{in }(0,T) \times \Omega \, ,
		\\
		(\hat{D}_\hom \nabla u - (\hat{W}_\hom + \hat{V}_\hom)u)\cdot n & = 0 && \textrm{on }(0,T) \times \partial \Omega \, ,
		\\
		u(0) &= u^{\init} && \textrm{in } \Omega \, ,
	\end{align}
\end{subequations}
where the substitute diffusion coefficient $\hat{D}_\hom$, the counter flow velocity $\hat{W}_\hom$ and the advection velocity $\hat{V}_\hom$ are given by 
\begin{subequations}
	\begin{align}\label{eq:HomogenisedCoefficientSubCoord}
		\begin{aligned}
			(\hat{D}_\hom&)_{ij}\tx \coloneqq \intS \Theta^{-1}\txs \hat{D}^*\txs \dy
			\\
			=&
			\intS \Theta^{-1}\txs \intYp (\An D \PnmT) \txsy (e_j + \nabla_y \hat{\zeta}_j\txsy) \cdot e_i \dys 
		\end{aligned}
		\\
		\begin{aligned}
			\hat{W}_\hom&\tx \coloneqq -\intS\hat{D}^*\txs \nabla_x (\Theta^{-1})\txs \ds
			\\
			=&
			-\intS \sum\limits_{j=1}^d \intYp (\An D \PnmT)\txsy (e_j + \nabla_y \hat{\zeta}_j\txsy) \dy \, \partial_{x_j} (\Theta^{-1})\txs \ds
		\end{aligned}
		\\
		\begin{aligned}
			\hat{V}_\hom\tx \coloneqq& \intS \hat{V}^*\txs \Theta^{-1}\txs \ds
			\\
			=&
			-\intS \Theta^{-1}\txs \intYp (\Jn D \PnmT)\txsy \nabla_y \hat{\zeta}_0\txsy \, . \hspace{3cm}
		\end{aligned}
	\end{align}
\end{subequations}
and $\hat{\zeta}_i$ for $i \in \{1, \dots, n\}$ is given by \eqref{eq:CellProblem:I:SubCoord} and $\hat{\zeta}_0$ by \eqref{eq:CellProblem:0:SubCoord}.

The weak form of \eqref{eq:Homogenised:SubCoord} is given by: Find $u \in L^2(0,T;H^1(\Omega)) \cap H^1(0,T;H^1(\Omega)')$ such that
\begin{subequations}\label{eq:WeakForm:Homogenised:SubCoord}
	\begin{align}\label{eq:WeakForm:Homogenised:SubCoord:1}
		\intTO \partial_t (u) \varphi + \hat{D}_\hom \nabla u \cdot \nabla \varphi - ((\hat{W}_\hom + \hat{V}_\hom) u ) \cdot \nabla \varphi 
		&=
		\intTO (F + G)\,\varphi \, ,
		\\\label{eq:WeakForm:Homogenised:SubCoord:2}
		u(0) &= u^\init \qquad \qquad \textrm{ in } L^2(\Omega) \hspace{3cm}
	\end{align}
\end{subequations}
for all $\varphi \in L^2(0,T;H^1(\Omega)$.

\begin{thm}\label{thm:HomEqTrafoCoord}
	Let $(\hun, \huo)$ be a solution of \eqref{eq:ts-limit-eq:transformed} and $u = \Theta \hun$. Then, $\un$ is the unique solution of \eqref{eq:WeakForm:PreHomogenised:SubCoord} and $u$ is the unique solution of \eqref{eq:WeakForm:Homogenised:SubCoord}.
\end{thm}

\begin{proof}
	First, we choose $\varphi_0 = 0$ in \eqref{eq:ts-limit-eq:transformed}, which yields
	\begin{align}
		\begin{aligned}\label{eq:HomogenisationProof:Trafo:1}
			&\intTOSYp (\An D \PnmT)\txsy (\nabla_x \hun\txs + \nabla_y \huo\txsy)\big) \\
			& \qquad \qquad \qquad \cdot 
			\nabla_y \varphi_1\txsy \dysxt
			\\
			&+
			\intTOSYp \An \txsy \partial_s \psin \txsy \hun\txs \cdot
			\nabla_y \varphi_1(t,x,s,y) \dysxt = 0 \,
		\end{aligned}
	\end{align}
	for every $\varphi_1 \in L^2((0,T) \times \Omega \times S ;H^1_\#(\Yp))$.
	After separating the $x$- and $y$-variables we obtain
	\begin{align}\label{eq:hu1=hzeta_i}
		u_1\txsy = \sum\limits_{i=1}^d \hat{\zeta}_i\txsy \partial_{x_i} u_0\txs + \hat{\zeta}_0\txsy u_0\txs \, ,
	\end{align}
	where $\hat{\zeta}_i \in L^\infty((0,T) \times \Omega \times (0,1);H^1_\#(\Yp)/\R)$, for $i \in \{1,\dots, n\}$, are the unique solution of the cell problem \eqref{eq:CellProblem:I:SubCoord} and $\hat{\zeta}_0 \in L^\infty((0,T) \times \Omega \times S;H^1_\#(\Yp))$ is the unique solution of
	\eqref{eq:CellProblem:0:SubCoord}.
	
	Now, we insert \eqref{eq:hu1=hzeta_i} into \eqref{eq:ts-limit-eq:transformed} and set $\varphi_1 = 0$, which gives \eqref{eq:WeakForm:PreHomogenised:SubCoord}
	\begin{align}
		\begin{aligned}\label{eq:homogenised/two-scale:Ju:u}
			&\intTO |\Yp| (\Jn \hun)\tx \partial_t \varphi_0\tx \dxt
			\\
			&+\intTOS \Big(\hat{D}^*\txs \nabla_x u_0\txs - \hat{U}^*\txs \hun \txs\Big)\ds \cdot \nabla \varphi_0\tx \dxt 
			\\
			&=
			\intTO (\hat{F}\tx +\hat{G}\tx) \varphi_0 \tx \dxt
		\end{aligned}
	\end{align}
	where $\Theta$, $\hat{D}^*, \hat{F}, \hat{G}$ are given by \eqref{eq:PreHomogenisedCoeffcients:SubCoord} and
	\begin{align}\label{eq:HomogenisationProof:Trafo:2}
		\hat{U}^*\txs \coloneqq -\intYp \An\txsy \big(D \PnmT\txsy \nabla_y \hat{\zeta}_0\txsy + \partial_s \psin \txsy \dy \,,
	\end{align}
	In order to simplify $\intS \hat{U}^* \hun$ in \eqref{eq:homogenised/two-scale:Ju:u}, we rewrite for $\xi \in \R^d$,
	\begin{align*}
		\begin{aligned}
			\An\xi &= \Jn \Pnm \xi
			=
			\Jn \xi + ( \1 -\Pn) \Jn\Pnm \xi
			= 
			\Jn \xi + \partial_y (y -\psin) \An \xi
			\\
			&=
			\Jn \xi + \left(\begin{array}{c}
				\nabla_y ((y -\psin)_1) \cdot \An \xi
				\\
				\vdots
				\\
				\nabla_y ((y -\psin)_d) \cdot \An \xi
			\end{array}\right) \,.
		\end{aligned}
	\end{align*}
	We note that $(y \mapsto (y -\psin\txsy)) \in H^1_\#(\Yp)$ for a.e.~$\txs \in (0,T) \times \Omega \times S$. 
	Thus, the cell problem 	\eqref{eq:CellProblem:0:SubCoord} shows for every $i \in \{1, \dots, d\}$ that
	\begin{align*}
		&\intYp (\nabla_y (y- \psin\txsy)_i) 
		\\
		&\qquad \cdot \An\txsy \big(D \PnmT\txsy \nabla_y \hat{\zeta}_0\txsy + \partial_s \psin \txsy \big) \dy = 0 \, .
	\end{align*}
	Combining, the previous two equations allows us to rewrite
	\begin{align}\label{eq:HomogenisationProof:Trafo:3}
		\hat{U}^*\txs = -\intYp \Jn\txs \big(D \PnmT\txsy \nabla_y \hat{\zeta}_0\txsy + \partial_s \psin \txsy \dy \, .
	\end{align}
	and, consequently, we get 
	\begin{align}\label{eq:proof:Simplifaction:Advection:001}
		\begin{aligned}
			\intS\hat{U}^*\txs \hun\txs \ds 
			= &-\intSYp \Jn\txs \big(D \PnmT\txsy \nabla_y \hat{\zeta}_0\txsy \hun \dys
			\\
			&\qquad - \intSYp \partial_s \psin \txsy \dy\, \Jn\txs \hun\txs \ds \, .
		\end{aligned}
	\end{align}
	Employing the fact that $(\Jn \hun) = |\Yp|^{-1}(\Theta \hun)$ is independent of $s$ and $\psin$ is $S$-periodic, we get
	\begin{align}\label{eq:proof:Simplifaction:Advection:002}
		\begin{aligned}
			\intSYp & \partial_s \psin \txsy \dy \, \Jn\txs \hun\txs \ds 
			= 
			\intSYp \partial_s \psin \txsy \dys \, (\Jn\hun)\tx \\
			& = \intYp \psin(t,x,1,y)-\psin(t,x,0,y) \dy \, (\Jn\hun)\tx = 0 \, .
		\end{aligned}
	\end{align}
	Thus, the second summand on the right-hand side of \eqref{eq:proof:Simplifaction:Advection:001} vanishes. This simplifies \eqref{eq:homogenised/two-scale:Ju:u} to \eqref{eq:WeakForm:PreHomogenised:SubCoord} for 
	$\hat{V}^*$ given by \eqref{eq:PreHomogenisedCoeffcients:SubCoord:V}.

	The reformulation of \eqref{eq:WeakForm:PreHomogenised:SubCoord} to \eqref{eq:WeakForm:Homogenised:SubCoord} for $\hat{D}_\hom$, $\hat{W}_\hom$ and $\hat{V}_\hom$ given by \eqref{eq:HomogenisedCoefficientSubCoord} is shown above.
\end{proof}

Following the argumentation of the simplification of \eqref{eq:HomogenisationProof:Trafo:2} to \eqref{eq:HomogenisationProof:Trafo:3} in the proof of Theorem \ref{thm:HomEqTrafoCoord} but using \eqref{eq:HomogenisationProof:Trafo:1} instead of the cell problem 
\eqref{eq:CellProblem:I:SubCoord}, we get
\begin{align}\label{eq:Simplifiaction:Two-scale-Limit:Trafo}
	\intSYp \An \big( D \PnmT\nabla_y \huo + \partial_s \psi_0\hun \big) 
	\cdot \nabla_x \varphi_0 
	=
	\intSYp \big(\Jn D \PnmT\nabla_y \huo + \Jn \partial_s \psin \hun \big)
	\cdot \nabla_x \varphi_0 \,.
\end{align}
The last summand on the right-hand side of \eqref{eq:Simplifiaction:Two-scale-Limit:Trafo} vanishes by the argumentation as in \eqref{eq:proof:Simplifaction:Advection:001}--\eqref{eq:proof:Simplifaction:Advection:002} and we can simplify \eqref{eq:ts-limit-eq:transformed:1} to:
\begin{align}
	\begin{aligned}\label{eq:ts-limit-eq:transformed-Simplified:1} 
		&\intT \langle \partial_t (\Theta \hun), \varphi_0 \rangle_{H^1(\Omega)', H^1(\Omega)}
		+
		\intTOSYp \Jn D \PnmT \big( \nabla_x \hun + \nabla_y \huo \big) 
		\cdot \nabla_x \varphi_0 
		\\
		&+
		\intTOSYp \big( \An D \PnmT(\nabla_x \hun + \nabla_y \huo ) 
		+
		\An \partial_s \psin \hun\big)
		\cdot \nabla_y \varphi_1 
		\\
		=&
		\intTOSYp \Jn \hfn \varphi_0
		+
		\intTOSG \Jn \|\PnmT \hat{\nu}\| \hgn \, \varphi_0 \,.
	\end{aligned}
\end{align}

\section{Back-transformation of the homogenised equation}\label{sec:Back-Trafo}
We transform the above derive two-scale limit problem and both homogenised problems as well as the cell problems back to the actual two-scale domain. As result, we obtain transformation-independent equations and cell problems defined in the actual pulsating reference cells.

The following lemma transfers the two-scale convergence of $\hue$ in a two-scale convergence for $\ue$.
\begin{lemma}\label{lem:Trafo:Two-scale:Convergence}
	Let $p \in (1,\infty)$ and assume that $\psie$ and $\psin$ satisfy Assumption \ref{ass:psie}. Let $\ue$ be a sequence in $L^p(\OeT)$ and $\hue\tx = \ue(t,\psie\tx)$ be the corresponding transformed sequence in $L^p((0,T) \times \Oe)$.
	Let $\un \in L^p(\Omega_0^T)$ and $\hun \in L^2((0,T) \times \Omega \times S \times \Yp)$. Then
	\begin{align*}
		&\widetilde{\ue} \tsw{p} \widetilde{\un}&& \textrm{ if and only if } &&\widetilde{\hue} \tsw{p} \widetilde{\hun} \, ,
		\\
		&\widetilde{\ue} \tss{p} \widetilde{\un}&& \textrm{ if and only if } &&\widetilde{\hue} \tss{p} \widetilde{\hun} \, ,
	\end{align*}
	where $\hun(t,x,s,y) = \un(t,x,s,\psin\txsy)$.
	Moreover, $\nabla \ue$ is a sequence in $L^p(\OeT)$ if and only if $\nabla_x \hue$ is sequence in $L^p((0,T) \times \Oe)$ and it holds
	\begin{align*}
		&\widetilde{\nabla_x \ue} \tsw{p} \one{\Yptxs}(y) \nabla_x \un\txs + \widetilde{\nabla_y \uo}\txsy &&\textrm{ if and only if }
		\\
		&\widetilde{\hue} \tsw{p} \one{\Yp}(y)\nabla_x \hun\txs + \widetilde{\nabla_y \huo}\txsy
	\end{align*}
	for $\hun\txs = \un\txs$ and $\huo\txsy = \uo(t,x,s,\psin\txsy) + \nabla_x \un\txs \cdot (\psin\txsy -y)$\,.
\end{lemma}
\begin{proof}
	The result is shown in \cite{AA23} under weaker Assumptions. 
	There the transformations have no separate time direction as in our case. The equivalence of the weak convergence of the functions in Lemma \ref{lem:Trafo:Two-scale:Convergence} is given by \cite[Theorem 3.8]{AA23} applied on the diffeomorphisms $\tx \mapsto (t, \psie\tx)$ and the limit transformations $(s,y) \mapsto \psin(t,x,s,y)$ with parameter $\tx \in (0,T) \times \Omega$. The equivalence for the strong two-scale convergence is shown in \cite[Theorem 3.14]{AA23}.
	The second part of Lemma \ref{lem:Trafo:Two-scale:Convergence} can be deduced similarly to \cite[Theorem 3.10]{AA23}. 
\end{proof}

\subsection{Back-transformation of the two-scale limit problem}
In this section, we transform the two-scale limit problem \eqref{eq:ts-limit-eq} back to the domain $\Omega_0^T$.

To derive some transformation independent result, we use the reformulation of \eqref{eq:ts-limit-eq:transformed} by \eqref{eq:ts-limit-eq:transformed-Simplified:1}.
We note that 
\begin{align*}
	\PnmT\txsy &= (\nabla_y \psin)^{-1}\txsy = (\nabla_y (\psinm))(t,x,s, \psin\txsy)
	\\
	&=
	\1 + (\nabla_y (\psinm -y)) (t,x,s, \psin\txsy) \,.
\end{align*}
We rewrite the macroscopic diffusion term of \eqref{eq:ts-limit-eq:transformed-Simplified:1} and transform in back by means of $\psin$, which gives
\begin{align}\label{eq:MacroAdvectionBackTrafo}
	\begin{aligned}
		&	\intYp \Jn D \PnmT \big(\nabla_x \hun + \nabla_y \huo \big) 
		\cdot \nabla_x \varphi_0
		\\
		&	=
		\intYp \Jn\txs D \Big( \nabla_x \hun\txs + (\nabla_y ((\psinm -y)\cdot \nabla_x \hun) (t,x,s, \psin\txsy) 
		\\
		&\qquad \qquad \qquad \qquad+ \PnmT\txsy \nabla_y \huo \big) \Big) \cdot \nabla_x \varphi_0\txs
		\dy
		\\
		&	=
		\intYptxs D \nabla_x \un\txs + D \nabla_y \uo\txsy 
		\cdot \nabla_x \varphi_0\txs \dy
	\end{aligned}
\end{align}
for $\un = \hun$ and 
\begin{align}\label{eq:Identifcation:hun1=un1}
	\uo\txsy = \huo(t,x,s,\psinm\txsy)) + (\psinm \txsy-y)\cdot \nabla_x \hun\txs \,.
\end{align}

Before we transform the microscopic advection term of \eqref{eq:ts-limit-eq:transformed-Simplified:1}, we integrate the $\Yp$ integral by parts and identify $\partial_s \psin$ with $\hvGn$ on $\Gamma$, which gives
\begin{align}
	\begin{aligned}\label{eq:MicroAdvectionTerm:PartIntegration}
		\intYp \An \partial_s \psin \cdot \nabla_y \varphi_1 
		&= 
		-\intYp \div_y (\An \partial_s \psin) \varphi_1 + \intG \An \partial_s \psin \cdot \hat{\nu} \varphi_1
		\\
		&=
		-\intYp \partial_s \Jn \varphi_1 + \intG \An \hvGn \cdot \hat{\nu} \varphi_1
	\end{aligned}
\end{align}
We use the identity
\begin{align*}
	\|\PnT\txsy \hat{\nu}\txsy\|^{-1}\PnmT \txsy\hat{\nu}\txsy = \nu(t,x,s,\psin\txsy) \, ,
\end{align*}
in order to transform the boundary integral on the right-hand side of \eqref{eq:MicroAdvectionTerm:PartIntegration}, which gives
\begin{align}\label{eq:MicroAdvectionBackTrafo}
	\begin{aligned}
		&\intYp \An \partial_s \psin \cdot \nabla_y \varphi_1 
		=
		-(\partial_s \Jn ) \Jnm \intYp \Jn \varphi_1 + \intGtxs \Jn \hvGn \cdot \PnmT \hat{\nu} \varphi_1
		\\
		&=
		-(\partial_s \Jn) |\Yp| |\Yp|^{-1} \Jnm \intYptxs \varphi_1(t,x,s,\psinm\txsy) \dy+ \intGtxs \Jn \hvGn \cdot \PnmT \hat{\nu} \varphi_1
		\\
		&=
		-\partial_s \Theta \Theta^{-1} \intYptxs \varphi_1(t,x,s,\psinm\txsy) \dy
		\\
		& \qquad + \intG \Jn\txs (\|\PnmT\hat{\nu}\| \hvGn)
		\txsy \cdot \nu(t,x,s,\psinm\txsy) \varphi_1\txsy \dd \sigma_y
		\\
		&=
		-\partial_s \Theta \Theta^{-1} \intYptxs \varphi_1(t,x,s,\psinm\txsy) \dy
		\\
		& \qquad + \intGtxs \vGn\txsy
		\cdot \nu \varphi_1(t,x,s,\psinm\txsy) \dd \sigma_y \, .
	\end{aligned}
\end{align}

Having \eqref{eq:MacroAdvectionBackTrafo} and \eqref{eq:MicroAdvectionBackTrafo}, the remaining transformation of \eqref{eq:ts-limit-eq:transformed-Simplified:1} follows directly and we obtain th equivalent weak form:

Find $(\un, \uo) \in L^2((0,T) \times S;H^1(\Omega)) \times L^2((0,T) \times \Omega \times S; H^1_\#(\Yptxs))$, with 
$(\Theta \un) \in L^2(0,T;H^1(\Omega)) \cap H^1(0,T;H^1(\Omega)')$ such that
\begin{subequations}\label{eq:ts-limit-eq}
	\begin{align}
		\begin{aligned}\label{eq:ts-limit-eq:1}
			&\intT \langle \partial_t (\Theta \un), \varphi_0 \rangle_{H^1(\Omega)', H^1(\Omega)}
			+
			\intTOSYptxs D \big( \nabla_x \un + \nabla_y \uo \big) 
			\cdot \nabla_x \varphi_0 
			\\
			&+
			\intTOS \left(\intYptxs D (\nabla_x \un + \nabla_y \uo )
			\cdot \nabla_y \varphi_1 - \partial_s \Theta \Theta \un \varphi_1 + \un \intGtxs \vGn \cdot \nu \varphi_1 \right)
			\\
			=&
			\intTOSYptxs \fn \varphi_0
			+
			\intTOS\intGtxs \gn \varphi_0
		\end{aligned}
		\\
		\begin{aligned}
			\label{eq:ts-limit-eq:2}
			(\Theta \un)(0) = \Theta(0,\cdot,0)\unin \, \qquad \textrm{ in } L^2(\Omega) \hspace{4cm}
		\end{aligned}
	\end{align}
\end{subequations}
for all $(\varphi_0, \varphi_1) \in L^2(0,T;H^1(\Omega)) \times L^2((0,T) \times \Omega \times S;H^1_\#(\Yptxs))$.

The following lemma is a direct consequence of the above integral transformations.
\begin{lemma}
	The weak forms \eqref{eq:ts-limit-eq:transformed} and \eqref{eq:ts-limit-eq} are equivalent in the sense that $(\hun, \huo)$ solves \eqref{eq:ts-limit-eq:transformed} if and only if $(\un, \uo)$ solves \eqref{eq:ts-limit-eq:transformed}. The solutions can be identified by $\hun = \un$ and \eqref{eq:Identifcation:hun1=un1}.
\end{lemma}

\subsection{Homogenised problems}
As in the untransformed setting, we separate the microscopic and macroscopic space variables first, which gives the homogenised problem \eqref{eq:PreHomogenised}. In order to separate the microscopic time variable, we reformulate \eqref{eq:Homogenised} again in terms of $u = \Theta \un = \Theta \hun$ by introducing a counter flow balancing the reformulation of the diffusive flux. The result is given by \eqref{eq:Homogenised}.

The coefficients of those homogenised problems can be identified with those of the corresponding equation before the transformation. Thus, they are the same equations and the weak form of \eqref{eq:PreHomogenised} and \eqref{eq:Homogenised} is the same as the weak form of \eqref{eq:PreHomogenised:SubCoord} and \eqref{eq:Homogenised:SubCoord} but without the $\hat{\cdot}$ on the unknowns, coefficients and data.

\subsubsection{Separation of the spatial variables}The separation of the spatial variables in \eqref{eq:ts-limit-eq} gives the homogenised equation
\begin{subequations}\label{eq:PreHomogenised}
	\begin{align}\label{eq:PreHomogenised:1}
		\partial_t (\Theta \hun) - \div( {D}^* \nabla \un - {V}^* \un) &= {F} + {G}&& \textrm{in } (0,T) \times \Omega \times S \, ,
		\\\label{eq:PreHomogenised:2}
		({D}^* \nabla \hun - {V}^* \un) \cdot n &= 0 && \textrm{in } (0,T) \times \partial \Omega \times S \, ,
		\\\label{eq:PreHomogenised:3}
		\partial_s (\Theta \un) &= 0 &&\textrm{in } (0,T) \times \Omega \times S \, ,
		\\\label{eq:PreHomogenised:4}
		\un(0) &= \un^{\init} &&\textrm{in } \Omega \,.
	\end{align}
\end{subequations}
where the porosity $\Theta$, the effective diffusion coefficient $D^*$, the advection velocity $V^*$ and the bulk source terms $F$ and $G$ are given by
\begin{subequations}\label{eq:PreHomogenisedCoeffcients}
	\begin{align}
		&\Theta\txs \coloneqq \intYptxs 1 \dy = |\Yptxs| \,,
		\\
		&D_{ij}^*\txs \coloneqq \int\limits_{\Yp} D (e_j + \nabla_y \zeta_j) \cdot e_i \,,
		\\\label{eq:PreHomogenisedCoeffcients:V}
		&V^*\txs \coloneqq \intYptxs -D \nabla_y \zeta_0 \,,
		\\
		&F\tx \coloneqq \intS \intYptxs \fn \,,
		\\
		&G\tx \coloneqq \intS \intGtxs \gn \,.
	\end{align}
\end{subequations}
The effective diffusion coefficient $D^*$ and the advection velocity $V^*$ depend on the solutions $\zeta_i$ for $i \in \{1,\dots, n\}$ and $\zeta_0$ of the following cell problems:
\begin{align}\label{eq:CellProblem:I}
	\begin{aligned}
		-\div( D (e_j + \nabla_y \zeta_j)) &= 0 &&\textrm{in } \Yptxs \,,
		\\
		D (e_j + \nabla_y \zeta_j) \cdot \nu &= 0 &&\textrm{on } \Gtxs \,,
		\\
		y &\mapsto \zeta_j &&Y-\textrm{periodic}
	\end{aligned}
\end{align}
and 
\begin{align}\label{eq:CellProblem:0}
	\begin{aligned}
		-\div( D \nabla_y \zeta_0)) &= \partial_s \Theta \Theta^{-1} &&\textrm{in } \Yptxs \,,
		\\
		(D \nabla_y \zeta_0 + \vGn) \cdot \hat{\nu} &= 0 &&\textrm{on } \Gtxs \,,
		\\
		y &\mapsto \hat{\zeta}_0 &&\textrm{Y-}periodic,
	\end{aligned}
\end{align}

We note that one can identify the solutions of \eqref{eq:CellProblem:I} and \eqref{eq:CellProblem:0} with the solutions of \eqref{eq:CellProblem:I:SubCoord} and \eqref{eq:CellProblem:0:SubCoord}, respectively, via
\begin{align*}
	&\hat{\zeta}_j \txsy = \zeta_j(t,x,s,\psin\txsy) + (\psin\txsy -y) \cdot e_j
	\\
	&\hat{\zeta}_0\txsy = \zeta_0(t,x,s,\psin\txsy) + (\psin\txsy -y) \cdot e_j
\end{align*}
and the effective coefficients and right-hand sides by
\begin{align*}
	\hat{D}^* = \hat{D} \, \qquad 	\hat{V}^* = V^* \, , \qquad \hat{F} = F \, , \qquad \hat{G} = G \, .
\end{align*}

\subsection{Separation of the time variables}
In order to formulate without microscopic time, we use the same substitution as in Section~\ref{subsec:SeparationTimeVariables}. As result, we obtain the homogenised limit problem \eqref{eq:Homogenised},
where the coefficients effective diffusion coefficient $D_\hom$, the velocity of the artificial counterflow $W_\hom$ and the advection velocity $V_\hom$ are given by 
\begin{subequations}\label{eq:HomogenisedCoefficient}
	\begin{align}
		{D}_\hom\tx \coloneqq& \intS \Theta^{-1}\txs {D}^*\txs \dy = \hat{D}_\hom\tx
		\\
		{W}_\hom\tx \coloneqq& -\intS{D}^*\txs \nabla_x (\Theta^{-1})\txs \ds = \hat{W}_\hom\tx
		\\
		{V}_\hom\tx \coloneqq& \intS {V}^*\txs \Theta^{-1}\txs \ds = \hat{V}_\hom\tx
	\end{align}
\end{subequations}

\begin{rem}
	In the case that the porosity is macroscopically constant, i.e.~ ${\nabla_x \Theta =0}$, the separation of the micro- and macroscopic time variables does not cause any artificial advective counter flow. 
	This can be also observed from the fact that ${W_{\hom} \tx = 0}$ if $\nabla_x \Theta = 0$.
	Consequently, one gets 
	\begin{align}
		\partial_t u -\div(D_\hom \nabla u + V_\hom u ) &= F + G && \textrm{in }(0,T) \times \Omega \, ,
		\\
		({D}_\hom \nabla u+ {V}_\hom)u)\cdot n & = 0 && \textrm{on }(0,T) \times \partial \Omega \, ,
		\\
		u(0) &= u^{\init} && \textrm{in } \Omega \, .
	\end{align}
\end{rem}

\begin{rem}
	In the case that $\Theta$ is not oscillating on the time interval $S$, i.e.~${\partial_s \Theta = 0}$, one can also avoid the introduction of the advective counter flow $V_\hom$. Instead, in this case \eqref{eq:PreHomogenised:3} implies that $\un$ is independent of $s \in S$ and one can separates the microscopic time variable in \eqref{eq:PreHomogenised} directly, which leads to 
	\begin{subequations}\label{eq:PreHomogenised:NoPulsTheta}
		\begin{align}\label{eq:PreHomogenised:NoPulsTheta:1}
			\partial_t (\Theta \un) - \div\Bigg(\intS D^*\ds \, \nabla \un - \intS V^* \ds \, \un \Bigg) &= {F} + {G}&& \textrm{in } (0,T) \times \Omega \, ,
			\\\label{eq:PreHomogenised:NoPulsTheta:2}
			\Bigg(\intS D^*\ds \, \nabla \un - \intS V^* \ds \, \un \Bigg) \cdot n &= 0 && \textrm{in } (0,T) \times \partial \Omega \, ,
			\\\label{eq:PreHomogenised:NoPulsTheta:3}
			\un(0) &= \un^{\init} &&\textrm{in } \Omega \,,
		\end{align}
	\end{subequations}
	where the effective diffusion coefficient ${D}^*$, the advection velocity ${V}^*$ as well as the porosity and the bulk source terms $F$ and $G$ are given as above.
\end{rem}

\section{Analysis of the advection velocity $V^*$}\label{sec:Con:V}
In order to understand the new advection velocity $V^*$, we investigate the cell problem \eqref{eq:CellProblem:0} and its solution. In particular, we want to understand for what kind of domain evolutions $V_\hom$ is zero and when the oscillations produce non trivial advection.

\subsection{Oscillating microstructure without producing advection}
Let us assume that the domain evolution undergoes an evolution in the time interval $[0,t_1]$ for a point $t_1 \in S$ and undergoes the same evolution in return on $[t_1 ,1]$ by a possible different speed.
\begin{lemma}\label{lem:ZeroAdvection}
	Let ${\tau : [0,t_1] \to [t_1,1]}$ for $t_1 \in S$ be an bijective Lipschitz continuous function such that
	\begin{align}\label{eq:EvolutionBackwards}
		\Yp(t,x, s) &= \Yp(x,s, \tau(s)) && \textrm{ for } s \in [0,t_1]\,,
		\\
		v_\Gamma(t,x, s) &= v_\Gamma(t,x, \tau(s)) \partial_s \tau(s)^{-1} && \textrm{ for } s \in (0,t_1)\, ,
	\end{align}
	Then, $V_\hom =0$ as well as $\int_S V^*\txs \ds = 0$\,.
\end{lemma}
For instance, if the evolution is symmetric with respect to $t_1 = 0.5$ one has
$\tau(s) = 1-s$, which leads to $\Yp(t,x, s) = \Yp(x,s, 1-s)$ for $s \in [0,t_1]$.
\begin{proof}
	From \eqref{eq:EvolutionBackwards}, one gets
	\begin{align*}
		\Theta\txs &= \Theta(t,x, \tau(s))\, ,
		\qquad
		\partial_s \Theta\txs = \partial_s \Theta(t,x, \tau(s)) \partial_s \tau(s)^{-1}&& \textrm{ for } s \in S\, ,
	\end{align*}
	Consequently, we can identify the solutions of the cell problems \eqref{eq:CellProblem:0} by
	\begin{align}\label{eq:TrafoSolutionCellProblem}
		\zeta_0(t,x, s,y ) &= \zeta_0(t,x,\tau(s),y) \partial_s \tau(s)^{-1} && \textrm{ for } s \in S\, .
	\end{align}
	We use this identification in order to compute $\int_S V^*\txs \ds$.
	We split the time integral at $t_1$ 
	\begin{align*}
		V_\hom\txs &= -\int\limits_{(0,t_1)} \Theta^{-1} \intYptxs D\zeta_0 
		- \int\limits_{(t_1,1)} \Theta^{-1} \intYptxs D \zeta_0\, .
	\end{align*}
	Transforming the second time integral by $\tau$ and using \eqref{eq:TrafoSolutionCellProblem} gives
	\begin{align*}
		\int\limits_{(t_1,1)} \Theta^{-1} \intYptxs D \zeta_0
		&	=
		\int\limits_{(0,t_1)} \Theta^{-1}(t,x,\tau(s)) \int\limits_{Y^*(t,x,\tau(s))} D\zeta_0(t,x,\tau(s),y) |\partial_s \tau(s)|\dy \ds
		\\
		&=
		-\int\limits_{(0,t_1)} \Theta^{-1}(t,x,s) \int\limits_{Y^*(t,x,s)} D\zeta_0(t,x,s,y) \dy \ds
	\end{align*}
	and, thus, $V_\hom = 0$. Analogously, we get $\int_S V^*\txs \ds$.
\end{proof}

The argumentation of the proof of Lemma \ref{lem:ZeroAdvection} can be generalised to more general situation. For instance, if $\Yptxs$ is given by an obstacle which moves periodically in $Y$ without rotating, one can deduce $V_\hom =0$.
However, as soon as there is some rotation of the obstacle or some other evolution one can get an non-trivial advection, which we see in the following.

\subsection{Oscillating microstructure producing advection}
We define the obstacle $\mathrm{S} = (- a,a) \times (- b, +b)$ for $0< a < b \leq 0.1$.
We divide the time interval $S$ in the subintervals $S_1 \coloneqq \left[0,\tfrac{1}{4}\right)$, $S_2\coloneqq \left[\tfrac{1}{4},\tfrac{2}{4}\right)$, $S_3 \coloneqq \left[\tfrac{2}{4},\tfrac{3}{4}\right)$, $S_4 \coloneqq \left[\tfrac{3}{4},\tfrac{4}{4}\right]$.

We consider the following motion of the obstacle $\mathrm{S}$, which is illustrated in Figure~\ref{fig:CellEvolution}. At the start, we assume that the centre $\mathrm{S}$ is positioned at $x_0 = \left(\tfrac{1}{4},\tfrac{1}{2}\right)^\top$
Then, the obstacle moves in the positive $e_1$ direction by $1/2$ during $S_1$. During the second time interval the obstacle rotates by an angle $\pi/2$ around its centre. In the third time interval it moves back in the negative $e_1$ direction by $1/2$. Finally, it rotates back by an angle of $\pi/2$.
\begin{figure}[ht]
	\centering
	\includegraphics[width=1\linewidth]{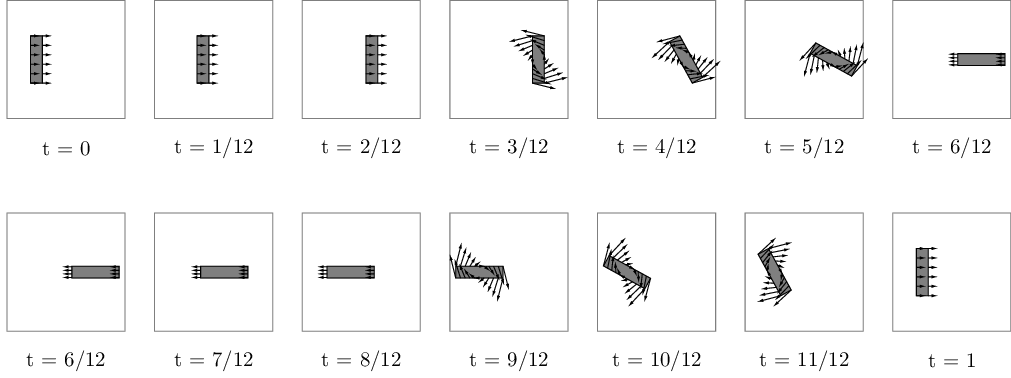}
	\caption{Illustration of an oscillating microstructure over the periodicity time interval $S$ with boundary velocity indicated by arrows. This geometry evolution leads to an advection velocity in the positive $e_1$ direction in the homogenised model}
	\label{fig:CellEvolution}
\end{figure}

We note that during this cell evolution the porosity $\Theta$ is constant. Thus, we have $V_\hom = \Theta(t,x) \int_S \int_{\Yptxs} D \zeta_0\txs \ds$.
In the following we show that this movement causes an advective velocity in the positive $e_1$ direction, i.e.~$\int_S V^*(t,x,s) = \lambda e_1$ for some $\lambda >0$ which depends on $a$ and $b$.

We note that up to a shift of the coordinate system by $\tfrac{1}{2}e_1$, we have
$Y^*(s) = Y^*(\tau(s))$ and $v_\Gamma(s) = -v_\Gamma(\tau(s))$ for $s \in \left(\tfrac{1}{4}, \tfrac{2}{4}\right)$ and $\tau(s) = 1- \left(s-\tfrac{1}{4}\right)$. Note that the shift of the coordinate system is not relevant for the computation of $V^*\txs$. Thus, arguing as in the proof of Lemma \ref{lem:ZeroAdvection}, we obtain
\begin{align*}
	\int\limits_{\left(\tfrac{1}{4}, \tfrac{2}{4}\right)}\intYptxs D \zeta_0 = -\int\limits_{\left(\tfrac{3}{4}, \tfrac{4}{4}\right)} \intYptxs D \zeta_0 \,
\end{align*}
and, thus,
\begin{align}\label{eq:EffectVelocity:Example-IntervallReduction}
	\int\limits_S V^* = -\int\limits_{\left(\tfrac{0}{4}, \tfrac{1}{4}\right)}\intYptxs D \zeta_0 - \int\limits_{\left(\tfrac{2}{4}, \tfrac{3}{4}\right)}\intYptxs D \zeta_0 \,.
\end{align}
For the remaining time intervals, we can insert $v_\Gamma = 2 e_i$ for $s \in\left(\tfrac{0}{4}, \tfrac{1}{4}\right)$ at the interface $\Gamma\txs$ and the cell problem can be simplified to
\begin{align*}
	\begin{aligned}
		-\div( D \nabla_y \zeta_0)) &= 0 &&\textrm{in } \Yptxs \,,
		\\
		(D \nabla_y \zeta_0 + 2 e_i) \cdot \hat{\nu} &= 0 &&\textrm{on } \Gtxs \,,
		\\
		y &\mapsto \hat{\zeta}_0 &&Y-\textrm{periodic},
	\end{aligned}
\end{align*}
One can proceed similarly for $s \in\left(\tfrac{2}{4}, \tfrac{3}{4}\right)$.

By the symmetry of the domain $\Yptxs$ with respect to $\{x_2 = 0.5\}$, one gets 
\begin{align}
	\begin{aligned}\label{eq:Expl-Comp-Coeffcient}
		\intYptxs D \nabla_y \zeta_0 &= -\lambda_1 e_i &&\textrm{for } s \in\left(\tfrac{0}{4}, \tfrac{1}{4}\right) \, ,
		\\
		\intYptxs D \nabla_y \zeta_0 &= \lambda_2 e_i&&\textrm{for } s\in\left(\tfrac{2}{4}, \tfrac{3}{4}\right) 
	\end{aligned}
\end{align} for some $\lambda_1\, ,\, \lambda_2 >0$.
Numerical computations, for fixed parameter $a,b$ with $a<b$ shows that $\lambda_1 > \lambda_2$. 
In order to convince ourselves without numerical computations that $\lambda_1 > \lambda_2$, we consider the case $a = \tfrac{1}{2}$, which is excluded since it leads to clogging, but it allows the explicit analytical computation of a solution. One gets $D\nabla_y\zeta = -2 e_1$ for $s \in\left(\tfrac{0}{4}, \tfrac{1}{4}\right)$ and $D\nabla_y\zeta = 0$ for $s \in\left(\tfrac{2}{4}, \tfrac{3}{4}\right)$. Inserting \eqref{eq:Expl-Comp-Coeffcient} in \eqref{eq:EffectVelocity:Example-IntervallReduction} shows
\begin{align*}
	\intS V^*\txs \ds = \lambda e_1
\end{align*} for $\lambda > 0$.

\bibliographystyle{plain}
\bibliography{pulsating}
\end{document}